\documentclass[12pt]{article}

\usepackage[affil-it]{authblk}
      \usepackage{latexsym}
         \usepackage[reqno, namelimits, sumlimits]{amsmath}
         \usepackage{amssymb, amsfonts}
         \usepackage{amsthm}

 \newtheorem{theorem}{Theorem}[section]
 
 \newtheorem{definition}[theorem]{Definition}
 
 \newtheorem{lemma}[theorem]{Lemma}
 
 \newtheorem{remark}[theorem]{Remark}
 
 \newtheorem{cor}[theorem]{Corollary}
 \newtheorem{pro}[theorem]{Proposition}

 \newcommand{\B}{\dot{B}^{-\frac{1}{4}}_{4,\infty}}
 \newcommand{\Bp}{\dot{B}^{s_p}_{p,\infty}}

\title{ Existence and Weak* Stability for the Navier-Stokes System with Initial Values in Critical Besov Spaces}

\author{T Barker
  \footnote{OxPDE, Mathematical Institute, University of Oxford, Oxford, UK. \\Email address: \texttt{tobiasbarker5@gmail.com}; }}
\date{ \today}

\begin{document}
\maketitle 

\begin{abstract}
In 2016, Seregin and \u Sver\'ak, conceived a notion of global in time solution (as well as proving existence of them) to the three dimensional Navier-Stokes equation with $L_3$ solenoidal initial data called 'global $L_3$ solutions'. A key feature of global $L_3$ solutions is continuity with respect to weak convergence of a sequence of solenoidal $L_3$ initial data.
The first aim of this paper is to show that a similar notion of ' global $\B$ solutions'  exists for solenoidal initial data in the wider critical space $\B$ and satisfies certain continuity properties with respect to weak* convergence of a sequence of solenoidal $\B$ initial data. This is the widest such critical space if one requires the solution to the Navier-Stokes equations minus the caloric extension of the initial data to be in the global energy class.

For the case of initial values in the wider class of  $\dot{B}^{-1+\frac{3}{p}}_{p,\infty}$ initial data ($p>4)$, we prove that for any $0<T<\infty$ there exists a solution to the Navier-Stokes system on $\mathbb{R}^3 \times ]0,T[$ with this initial data. We discuss how properties of these solutions imply a new regularity criteria for 3D weak Leray-Hopf solutions in terms of the norm $\|v(\cdot,t)\|_{\dot{B}^{-1+\frac{3}{p}}_{p,\infty}}$ (as well as certain additional assumptions).

 The main new observation of this paper, that enables these results, regards  the decomposition of homogeneous Besov spaces $\dot{B}^{-1+\frac 3 p}_{p,\infty}$.  This does not appear to obviously follow from the known real interpolation theory. 
\end{abstract}

\setcounter{equation}{0}

\section{Introduction}
In this paper, we consider the Cauchy problem for the Navier-Stokes system in the space-time domain $Q_S=\mathbb R^3\times ]0,S[$ ($0<S\leq \infty$) for  the vector-valued function $v=(v_1,v_2,v_3)=(v_i)$ and scalar function $q$, satisfying the equations
\begin{equation}\label{directsystem}
\partial_tv+v\cdot\nabla v-\Delta v=-\nabla q,\qquad\mbox{div}\,v=0
\end{equation}
in $Q_S$,
the boundary conditions
\begin{equation}\label{directbc}
v(x,t)\to 0\end{equation}
as $|x|\to\infty$ for all  $t\in  [0,S[$,
and the initial conditions
\begin{equation}\label{directic}
v(\cdot,0)=u_0(\cdot)
\end{equation}
In the recent paper \cite{sersve2016}, a notion of global in time solution to the Navier-Stokes equation was developed with $L_3$ solenoidal initial data called 'global $L_3$ solutions'.  A key feature of global $L_{3}$ solutions is as follows. Namely, if  $u^{(n)}$ are  global $L_{3}$ solutions corresponding the the initial datum $u^{(n)}_0$, and $u^{(n)}_0$ converge weakly in $L_{3}(\mathbb{R}^3)$ to $u_0$, then a suitable subsequence of $u^{(n)}$ converges to global $L_{3}$ solution $u$ corresponding to the initial condition $u_0$.
To explain the notation of global weak $L_3$-solutions in \cite{sersve2016} further, we introduce the notation $$S(t)u_0(x)=\int\limits_{\mathbb R^3}\Gamma(x-y,t)u_0(y)dy,
$$
where $\Gamma$ is the three dimensional heat kernel. Throughout this paper we will often write $V(x,t):=S(t)u_0(x)$. In \cite{sersve2016}, any global weak $L_3$-solution of the Navier-Stokes equation, with initial data $u_0\in L_{3}(\mathbb{R}^3)$, has the following structure. Namely,
\begin{equation}\label{globalL3solution}
v(x,t)= V(x,t)+u(x,t),
\end{equation}
where $u$ is globally in the energy class in $\mathbb{R}^3 \times ]0,T[$ for any  finite $T>0$ and satisfies the global energy inequality 

\begin{equation}\label{globalL3energyinequality}
\|u(\cdot,t)\|_{L_{2}}^2+2\int\limits_{0}^t\int\limits_{\mathbb R^3} |\nabla u(x,t')|^2 dxdt'\leq 2 \int_{0}^t\int_{\mathbb R^3}(V\otimes u+V\otimes V):\nabla udxdt'.
\end{equation}
In \cite{sersve2016}, the crucial estimate 
\begin{equation}\label{globalL3pertrubedestimate}
\|u(\cdot,t)\|_{L_{2}}^2+2\int\limits_{0}^t\int\limits_{\mathbb R^3} |\nabla u(x,t')|^2 dxdt'\leq t^{\frac{1}{2}} C(\|u_0\|_{L_{3}(\mathbb{R}^3)}).
\end{equation} 
is proven for $u$, with $u$ and $V$ as in (\ref{globalL3solution})-(\ref{globalL3energyinequality}). This estimate plays a central role in \cite{sersve2016} in the proof of continuity of global weak $L_3$ solutions, with respect to the weak convergence of the initial data. 

A natural question concerns whether an analogous notion of  of global in time solutions (which we will refer to as 'global $X$ solutions' or sometimes $\mathcal{N}(X)$)  with initial data $u_0$ in  other critical\footnote{ We say $X$ is a critical space, if $u_0 \in X \Rightarrow \lambda u_0(\lambda\cdot) \in X$ and $\|u_{0\lambda}\|_{X}=\|u_0\|_{X}$} spaces  $X$ , that satisfy the properties

\begin{itemize}
\item \textbf{1) Global existence} for any $u_0\in X$ there exists a global in time solution in $\mathcal{N}(X)$; 
 
\item \textbf{2) Weak* stability} $$u_0^{(k)}\stackrel{*}{\rightharpoonup} u_0\,\,\,\,\,\textrm{in}\,\,X\Rightarrow $$
   $ u^{(k)}(\cdot, u_{0}^{(k)})\in \mathcal{N}(X)$ converges up to subsequence (in sense of distributions) to $u(\cdot, u_0)\in \mathcal{N}(X);$
 
 \end{itemize}

 As mentioned in \cite{barkerser16}, such properties are useful if one wants to show that critical norms of the Navier-Stokes equations tend to infinity at a potential blow up time. 
 Let us now describe this in more detail. Suppose  one wanted to prove that if $v$ is a weak Leray-Hopf solution on $\mathbb{R}^3\times ]0,\infty[$, with sufficiently regular initial data as well as a finite blow up time $T$, then necessarily
 \begin{equation}\label{necessarilylorentzhalf}
 \lim\limits_{t\uparrow T}\|v(\cdot,t)\|_{X}=\infty.
 \end{equation}
 Where $X$ is a critical space.
 One such strategy, given in \cite{Ser10} and subsequently used in \cite{Ser12},\cite{barkerser16blowup} and \cite{dallas}, for showing this is to assume for contradiction  that there exists $t_n\uparrow T$ with
\begin{equation}\label{contradictionassumption}
M:=\sup\limits_{n}\|v(\cdot,t_n)\|_{X}<\infty.
\end{equation}
The next step is to perform the rescaling
 \begin{equation}
 u^{(n)}(y,s)=\lambda_nv(x,t), \quad
p^{(n)}(y,s)=\lambda^2_nq(x,t),\,\,\,\,  u^{(n)}_0(y)=\lambda_nv(\lambda_ny,t_n),
\end{equation}
\begin{equation}x=\lambda_ny,\qquad t=T+\lambda^2_ns,\qquad \lambda_n=\sqrt{\frac {T-t_n}2}
\end{equation}
This gives that $(u^{(n)},p^{(n)})$ are solutions to the Navier-Stokes equations on $\mathbb{R}^3\times ]-2,0[$ with 
\begin{equation}\label{boundedXdata}
\sup\limits_n\|u^{(n)}_0\|_{X}=\|v(\cdot,t_n)\|_{X}=
M.
\end{equation}
The final part of the strategy is to obtain a non-trivial ancient solutions to the Navier-Stokes equations and to then attempt to apply a Liouville type theorem to it, based on backward uniqueness and unique continuation for parabolic operators developed in \cite{ESS2003}. The motivation for considering 'global $X$ solutions', with properties \textbf{1) Global existence- 2) Weak* stability}, is that if such solutions exist then they are particularly useful in obtaining a non-trivial ancient solution in the above strategy. 

Unfortunately, the method for proving the existence of global $L_{3}$ solutions, with properties \textbf{1) Global Existence- 2) Weak* stability}, breaks down for $L^{3,\infty}$ initial data. In \cite{barkerser16}-\cite{barkersersverak16}, this difficulty was overcome. In particular in \cite{barkerser16}- \cite{barkersersverak16}, a notion of 'global weak $L^{3,\infty}$ solutions', satisfying \textbf{1) Global existence - 2) Weak* stability} and having the structure (\ref{globalL3solution})-(\ref{globalL3energyinequality}), was conceived. In \cite{barkerser16}-\cite{barkersersverak16}  the key is establishing
that for any global $L^{3,\infty}$ solution $v$, with  solenoidal initial data $u_0 \in L^{3,\infty}(\mathbb{R}^3)$, we have
\begin{equation}\label{globalweakL3pertrubedestimate}
\|v(\cdot,t)-S(t)u_0\|_{L_{2}}^2+2\int\limits_{0}^t\int\limits_{\mathbb R^3} |\nabla v(x,t')-S(t)u_0|^2 dxdt'\leq t^{\frac{1}{2}} C(\|u_0\|_{L^{3,\infty}(\mathbb{R}^3)}).
\end{equation}
A key part in showing this uses  that for any $N>0$ any divergence free function $u_0$ in $L^{3,\infty}(\mathbb{R}^3)$ can be decomposed into two divergence free pieces $\bar{u_0^N}$ and $\widetilde{u_0^N}$ satisfying
\begin{equation}\label{widetildeL2est}
\|\widetilde{u_0^N}\|_{L_{2}}^2\leq CN^{-1} \|u_0\|_{L_{3,\infty}}^3
\end{equation} 
and
\begin{equation}\label{barLp}
\|\bar{u_0^N}\|_{L_{p}}^p\leq CN^{p-3} \|u_0\|_{L_{3,\infty}}^3
\end{equation}
for any $3<p$. This, together with appropriate decompositions of the Navier-Stokes equations (inspired by \cite{Calderon90}), imply (\ref{globalweakL3pertrubedestimate}).

Suppose one attempts constructs a global $X$ solution to the Navier-Stokes equation, with $X$ being a critical space, having the structure (\ref{globalL3solution})-(\ref{globalL3energyinequality}). To ensure the finiteness of the right hand side of the energy inequality (\ref{globalL3energyinequality}), one should have that for any $u_0\in X$ that $$S(t)u_0 \in L_{4,loc}(0,\infty; L_{4}(\mathbb{R}^3)).$$ Hence, it is natural to consider $X$ such that for any $0<T<\infty$ and $0<\epsilon<T$:
\begin{equation}\label{L4semigroup}
\int\limits_{\epsilon}^T \|S(t)u_0\|_{L_{4}} dt\leq c(T,\epsilon) \|u_0\|_{X}.
\end{equation}
This implies that
\begin{equation}\label{L4t=1embedding}
\|S(1)u_0\|_{L_{4}}\leq c\|u_0\|_{X}.
\end{equation}
For $\lambda>0$ and $u_{0\lambda}(x):=\lambda u_{0}(\lambda x)$, it is clear that $S(1)u_{0\lambda}(x)= \lambda S(\lambda^2)u_0(\lambda x)$.
Hence, (\ref{L4t=1embedding}) implies that for any $\lambda>0$
\begin{equation}\label{optimalbesov}
\lambda^{\frac 1 4} \|S(\lambda^2)u_0\|_{L_{4}}= \|S(1)u_{0\lambda}\|_{L_{4}}\leq C\|u_{0\lambda}\|_{X}=C\|u_{0}\|_{X}.
\end{equation}
Hence,
\begin{equation}\label{optimalBesovembedding}
\sup_{s>0} s^{\frac 1 8}\|S(s)u_0\|_{L_{4}}\leq C\|u_{0}\|_{X}.
\end{equation}
In particular $X\hookrightarrow \B(\mathbb{R}^3)$. Thus, if one seeks a global $X$ solution, satisfying the requirements \textbf{1) Global existence- 2) Weak* Stability} and having the structure (\ref{globalL3solution})-(\ref{globalL3energyinequality}), $X=\B(\mathbb{R}^3)$ is the widest critical space for such a possibility.

 
 The aim of this paper is to show that, for any divergence free initial data in  $\B$,
  there exists 'global $\B$ solution' to the Navier-Stokes equations (\ref{directsystem})-(\ref{directic}), which satisfies the requirements \textbf{1) Global existence- 2) Weak* Stability}. 
  To the best of the author's knowledge, there is currently no notion of solutions to the Navier-Stokes equations with  arbitrary $\B$  solenodial initial data. Note that under certain smallness conditions on the $\B$ data, solutions were constructed by means of the Banach contraction principle, in \cite{Plan1996} and \cite{cannone1997}.  Let us mention that the recent preprint \cite{bradshawtsai} implies that if  $u_0 \in \B$ is discretely self similar \footnote{ the discretely self similar solutions in \cite{bradshawtsai} are shown to exist for any discretely self similar data in $\dot{B}^{-1+\frac 3 p}_{p,\infty}(\mathbb{R}^3)$ with $3<p<6$.}, then there exists a discretely  self similar solution to the Navier-Stokes equation. These solutions will also belong to our class of global $\B$ solutions.   
 
 

 Before giving the definition of 'global $\B(\mathbb{R}^3)$ solutions', we provide some relevant definitions and notation: 

$J$ and $\stackrel{\circ}J{^1_2}$ are the completion of the space
$$C^\infty_{0,0}(\mathbb R^3):=\{v\in C^\infty_0(\mathbb R^3):\,\,{\rm div}\,v=0\}$$
with respect to $L_2$-norm and the Dirichlet integral
$$\Big(\int\limits_{\mathbb R^3} |\nabla v|^2dx\Big)^\frac 12,$$
 correspondingly. Additionally, we define the space-time domains $Q_T:=\mathbb R^3\times ]0,T[$ and $Q_\infty:=\mathbb R^3\times ]0,\infty[$.
 
  For arbitrary vectors $a=(a_{i}),\,b=(b_{i})$ in $\mathbb{R}^{n}$ and for arbitrary matrices $F=(F_{ij}),\,G=(G_{ij})$ in $\mathbb{M}^{n}$ we put
 $$a\cdot b=a_{i}b_{i},\,|a|=\sqrt{a\cdot a},$$
 $$a\otimes b=(a_{i}b_{j})\in \mathbb{M}^{n},$$
 $$FG=(F_{ik}G_{kj})\in \mathbb{M}^{n}\!,\,\,F^{T}=(F_{ji})\in \mathbb{M}^{n}\!,$$
 $$F:G=
 F_{ij}G_{ij},\,|F|=\sqrt{F:F}.$$

\begin{definition}\label{globalBesovinf}

 We say that  $v$ 
 is a  weak $\dot{B}^{-\frac 1 4}_{4,\infty}$-solution to Navier-Stokes IVP in $Q_T$ (with $0<T<\infty$) 
  if 
\begin{equation}\label{weaksolutionsplitting}
v=V+u,
\end{equation}
with $u\in L_{\infty}(0,T; J)\cap L_{2}(0,T;\stackrel{\circ} J{^1_2})$  and there exists an $\alpha>0$ such that  
\begin{equation}\label{perturbedspaces}
\sup_{0<t<T}\frac{\|u(\cdot,t)\|_{L_{2}}^2}{t^{\alpha}}<\infty.
\end{equation} 
Additionally, it is required that there exists a $q\in L_{\frac 32, {\rm loc}}(Q_T)$ such that $u$ and $q$  satisfy the perturbed Navier- Stokes system in the sense of distributions:
\begin{equation}\label{perturbdirectsystem}
\partial_t u+v\cdot\nabla v-\Delta u=-\nabla q,\qquad\mbox{div}\,u=0
\end{equation}
in $Q_T$.
Furthermore, it is required that for any $w\in L_{2}$:
\begin{equation}\label{vweakcontinuity}
t\rightarrow \int\limits_{\mathbb R^3} w(x)\cdot u(x,t)dx
\end{equation}
is a continuous function on  $[0,T].$
Moreover, $u$ satisfies the energy inequality:
\begin{equation}\label{venergyineq}
\|u(\cdot,t)\|_{L_{2}}^2+2\int\limits_{0}^t\int\limits_{\mathbb R^3} |\nabla u(x,t')|^2 dxdt'\leq$$$$\leq 2 \int_{0}^t\int_{\mathbb R^3}(V\otimes u+V\otimes V):\nabla udxdt'
\end{equation}
for all $t\in [0,T]$.

Finally, it is required that $v$ and $q$ satisfy the local energy inequality.
Namely, for almost every $t\in ]0,T[$ the following inequality holds for all non negative functions $\varphi\in C_{0}^{\infty}(Q_T)$:
\begin{equation}\label{localenergyinequality}
\int\limits_{\mathbb R^3}\varphi(x,t)|v(x,t)|^2dx+2\int\limits_{0}^t\int\limits_{\mathbb R^3}\int\varphi |\nabla v|^2 dxdt^{'}\leq$$$$\leq
\int\limits_{0}^{t}\int\limits_{\mathbb R^3}|v|^2(\partial_{t}\varphi+\Delta\varphi)+v\cdot\nabla\varphi(|v|^2+2q) dxdt^{'}.
\end{equation}

We define $v$ to be a global $\dot{B}^{-\frac 1 4}_{4,\infty}(\mathbb{R}^3)$ solution if it is a weak $\B$-solution in $Q_T$ for any $0<T<\infty$.
\end{definition}
\begin{remark}\label{holderuinteg}
The role of the requirement (\ref{perturbedspaces}) is to ensure that the right hand side of the energy inequality (\ref{venergyineq}) is finite. See Lemma \ref{integabilitynonlinearity} for more details. This therefore ensures that the function $u$ satisfies the initial  condition in the strong $L_2$-sense, i.e.,
$u(\cdot,t)\to 0$ in $L_2$.
\end{remark}
 Next us state the main results of this paper.

\begin{theorem}\label{weak stability}
Let 
$u_{0}^{(k)}\stackrel{*}{\rightharpoonup} u_0$ in $\dot{B}^{-\frac 1 4}_{4,\infty}$ and let $v^{(k)}$ be a sequence of a global weak $\dot{B}^{-\frac 1 4}_{4,\infty}$-solutions to the Cauchy problem for the Navier-Stokes system with initial data $u_0^{(k)}$. Then there exists a subsequence still denoted $v^{(k)}$ that converges to a global weak $\dot{B}^{-\frac 1 4}_{4,\infty}$-solution $v$ to the Cauchy problem for the Navier-Stokes system with initial data $u_0$,
in the sense of distributions.
\end{theorem}
\begin{cor}\label{existenseglobal} There exists at least one global weak $\dot{B}^{-\frac 1 4}_{4,\infty}$-solution to the Cauchy problem 
(\ref{directsystem})-(\ref{directic}).
\end{cor}

Let us state our main observation, regarding decomposition of homogeneous Besov spaces, that enables us to show the global existence of global $\B$ solutions and that the property \textbf{2) Weak* stability} holds true for them.
 Note that from this point onwards, for $p_0>3$, we will denote $$s_{p_0}:=-1+\frac{3}{p_0}<0.$$
Moreover, for $ 2<\alpha\leq 3$ and $p_{1}>\alpha$, we define $$s_{p_1,\alpha}:= -\frac{3}{\alpha}+\frac{3}{p_1}<0.$$

\begin{pro}\label{besovinfinityinterpolationtheorem}
Suppose that $3<p<\infty$, 
\begin{equation}\label{Besovassumption}
g\in \dot{B}^{s_{p}}_{p,\infty}(\mathbb{R}^3)
\end{equation}
and
 $\textrm{div}\,\,g=0$ in sense of tempered distributions. 
  Then the above assumptions imply that there exists $p<p_2<\infty$, $0<\delta_{2}<-s_{p_2}$, $\gamma_{1}:=\gamma_{1}(p)>0$ and $\gamma_{2}:=\gamma_{2}(p)>0$ such that   for any $N>0$ there exists divergence free tempered distributions 
$\bar{g}^{N}\in \dot{B}^{s_{p_2}+\delta_{2}}_{p_2,p_2}(\mathbb{R}^3)\cap \dot{B}^{s_{p}}_{p,\infty}(\mathbb{R}^3)$ and $\widetilde{g}^{N}\in L_2(\mathbb{R}^3)\cap \dot{B}^{s_{p}}_{p,\infty}(\mathbb{R}^3) $ with the following properties. Namely,
 
\begin{equation}\label{Vdecomp1}
g= \bar{g}^{N}+\widetilde{g}^{N},
\end{equation}
\begin{equation}\label{barV_1est}
\|\bar{g}^{N}\|_{\dot{B}^{s_{p_2}+\delta_{2}}_{p_2,p_2}}\leq N^{\gamma_{1}}C(p, \|g\|_{\dot{B}^{s_{p}}_{p,\infty}}),
\end{equation}
\begin{equation}\label{barV_2est}
\|\widetilde{g}^{N}\|_{L_2}\leq N^{-\gamma_{2}}C(p,\|g\|_{\dot{B}^{s_{p}}_{p,\infty}}).
\end{equation} 
Furthermore, 
\begin{equation}\label{barV_1est.1}
\|\bar{g}^{N}\|_{\dot{B}^{s_p}_{p,\infty}}\leq C(p, \|g\|_{\dot{B}^{s_{p}}_{p,\infty}}),
\end{equation}
\begin{equation}\label{barV_2est.1}
\|\widetilde{g}^{N}\|_{\dot{B}^{s_p}_{p,\infty}}\leq C(p, \|g\|_{\dot{B}^{s_{p}}_{p,\infty}}).
\end{equation} 
\end{pro}
This refines previous decompositions for homogeneous Besov spaces, proven by the author in \cite{barker2017}. This is the main new observation of this paper. 
Unlike the decompositions in \cite{barker2017}, it is not clear if related decompositions are obtainable by the known real interpolation theory of homogeneous Besov spaces.  
 
 Once Proposition \ref{besovinfinityinterpolationtheorem} is established, one can argue in a similar manner to the case of global weak $L^{3,\infty}$ solutions to obtain improved decay properties of $u(x,t):=v(x,t)-V(x,t)$ near the initial time, which we will state as a Lemma. Related properties have also been exploited by the author in establishing weak strong uniqueness results for the Navier-Stokes equation, see \cite{barker2017}.
 \begin{lemma}\label{venergyest}
Let $u$, $v$ and $u_0$ be as in Definition \ref{globalBesovinf}. Let $\gamma_{1}$, $\gamma_{2}$, $\delta_{2}$ and $p_2$ be as in Proposition \ref{besovinfinityinterpolationtheorem}. Then there exists a $\beta(\gamma_{1},\gamma_{2},\delta_{2})>0$ such that the following estimate is valid for any $0<t<T$:
\begin{equation}\label{venergybddscaled}
\|u(\cdot,t)\|_{L_{2}}^2+\int\limits_0^t\int\limits_{\mathbb{R}^3} |\nabla u|^2dxdt'\leq
C(T, \|u_0\|_{\B},\delta_{2})t^{\beta}.
\end{equation}
\end{lemma}
Once this Lemma is established, Theorem \ref{weak stability} and Corollary \ref{existenseglobal} follow by similar arguments presented for global $L^{3,\infty}$ solutions in \cite{barkerser16}-\cite{barkersersverak16}.

We will also prove some conditional uniqueness and regularity statements for global weak  $\B$-solutions.

 \begin{pro}\label{besovsmallness}
 Let $u_{0}\in \B$ be divergence free in the sense of tempered distributions.
 There exists an $\varepsilon>0$ such that if
 \begin{equation}\label{initialdatacondition}
 \sup_{0<t<T} t^{\frac 1 8} \|S(t)u_0\|_{L_{4}}<\varepsilon 
 \end{equation}
 then all global $\B$ solutions, with initial data $u_0\in \B$, coincide on $Q_{T}$.
 
 \end{pro}


 Unfortunately, the assumptions of Proposition \ref{besovsmallness} do not hold for arbitrary divergence free initial data in $\B$. In particular, for  large minus one solenoidal initial data. For this case, the conjectures in \cite{jiasverak2014} and \cite{jiasverak2015} suggest that there may be non uniqueness for global $\B$ solutions, even for a short time.
 
This contrasts drastically with the global $L_{3}$ solutions to the Navier-Stokes equations, which are constructed in \cite{sersve2016} for any $L_{3}(\mathbb{R}^3)$ solenoidal initial data. In \cite{sersve2016}, short time uniqueness and regularity results are shown to hold for arbitrary $L_{3}(\mathbb{R}^3)$ solenoidal initial data. A key feature in showing this is that for any $L_{3}(\mathbb{R}^3)$ solenoidal initial data, there exists a local in time mild solution to the Navier-Stokes equations. This was proven in \cite{Kato1984}, which crucially makes use of the fact that  
\begin{equation}\label{kato5classshrink}
\lim_{T\rightarrow 0}\sup_{0<t<T} t^{\frac 1 5} \|S(t)u_0\|_{L_{5}(\mathbb{R}^3)}=0,
\end{equation}
for any $ u_0\in L_{3}(\mathbb{R}^3)$. Such a property follows from the density of test functions in $L_{3}(\mathbb{R}^3).$

Although  local in time mild solutions to the Navier-Stokes equations have been constructed for wide classes of initial data (see for example,  \cite{cannone1997}, \cite{FujKato1964}, \cite{GigaMiy1989}, \cite{KoTa2001},    \cite{KozYam1994},  \cite{Plan1996}, and \cite{Taylor1992}) it is not known if they can be constructed for arbitrary $\B$ solenoidal initial data.
The main obstacle is that, unlike $L_{3}(\mathbb{R}^3)$, test functions are not dense in $\B$.

\begin{subsection}{Solutions to the Navier-Stokes Equations with Initial Values in $\dot{B}^{s_p}_{p,\infty}(\mathbb{R}^3)$ and Applications to Regularity Criteria}
For solenodial initial data in $u_0\in \dot{B}^{s_p}_{p,\infty}(\mathbb{R}^3)$ with $4<p<\infty$, we do not know if there is a solution to the Navier-Stokes equations that satisfies \textbf{1) Global existence-2) Weak* Stability}.  Proposition \ref{besovinfinityinterpolationtheorem}, together with ideas from \cite{Calderon90} and a Lemma from \cite{dallas}, is sufficient to give a notion of solution of the Navier-Stokes equations with  solenodial $\Bp$ initial data. To the best of our knowledge, there is no previous notion of solution to the Navier-Stokes equations with arbitrary solenodial $\Bp$ initial data. Note that under certain smallness conditions on the $\Bp$ data, solutions were constructed by means of the Banach contraction principle, in \cite{Plan1996} and \cite{cannone1997}. We also mention an application of this solution in providing a new regularity criteria for weak Leray-Hopf solutions of the Navier-Stokes equations.

Let $3<p<\infty$ and suppose $u_0 \in \dot{B}^{s_p}_{p,\infty}$ is a divergence free tempered distribution. 
Let $u_0= \bar u_0^{N}+\widetilde u^{N}_0$ denote the splitting from Proposition \ref{besovinfinityinterpolationtheorem}. In particular, $p<p_2<\infty$, $0<\delta_{2}<-s_{p_2}$, $\gamma_{1}:=\gamma_{1}(p)>0$ and $\gamma_{2}:=\gamma_{2}(p)>0$ are such that   for any $N>0$ there exists  weakly divergence free functions 
$\bar{u}_0^{N}\in \dot{B}^{s_{p_2}+\delta_{2}}_{p_2,p_2}(\mathbb{R}^3)\cap \Bp(\mathbb{R}^3)$ and $\widetilde{u}^{N}_0\in L_2(\mathbb{R}^3)\cap \Bp(\mathbb{R}^3) $ with
 
\begin{equation}\label{Vdecomp1intro}
u_0= \bar{u}_0^{N}+\widetilde{u}_0^{N},
\end{equation}
\begin{equation}\label{barV_1estintro}
\|\bar{u}_0^{N}\|_{\dot{B}^{s_{p_2}+\delta_{2}}_{p_2,p_2}}\leq N^{\gamma_{1}}C(p, \|u_0\|_{\Bp}),
\end{equation}
\begin{equation}\label{barV_2estintro}
\|\widetilde{u}_0^{N}\|_{L_2}\leq N^{-\gamma_{2}}C(p,\|u_0\|_{\Bp}).
\end{equation} 
Furthermore, 
\begin{equation}\label{barV_1est.1intro}
\|\bar{u}_0^{N}\|_{\Bp}\leq C(p, \|u_0\|_{\Bp}),
\end{equation}
\begin{equation}\label{barV_2est.1intro}
\|\widetilde{u}_0^{N}\|_{\Bp}\leq C(p, \|u_0\|_{\Bp}).
\end{equation} 
\begin{theorem}\label{Bpsolutionexistence}
Let $u_0 \in \dot{B}^{s_p}_{p,\infty}(\mathbb{R}^3)$ be divergence free, in the sense of tempered distributions, and let $4<p<\infty$.
For any  finite $T>0$ there exists an $N=N(\|u_0\|_{\dot{B}^{s_p}_{p,\infty}},T)>0$ such that the following conclusions hold. 
  In particular, there exists a solution to Navier-Stokes IVP in $Q_T$ 
  such that 
\begin{equation}\label{weaksolutionsplittingBp}
v=W^{N}+u^{N}.
\end{equation}
Here, $W^{N}$ is a mild solution to the Navier-Stokes equations, with initial data $\bar{u}_{0}^{N}\in\dot{B}^{s_{p_2}+\delta_{2}}_{p_2,p_2}(\mathbb{R}^3)$, such that
\begin{equation}\label{mildBesovestimate}
\sup_{0<t<T} t^{\frac{p_2}{2}-\frac{\delta_{2}}{2}}\|W^N(\cdot,t)\|_{L_{p_2}}\leq  N^{\gamma_{1}}C(p, \|u_0\|_{\dot{B}^{s_{p}}_{p,{\infty}}}).
\end{equation}
Furthermore, $u^{N}\in L_{\infty}(0,T; J)\cap L_{2}(0,T;\stackrel{\circ} J{^1_2})$.  
Additionally,  there exists a $q^{N}\in L_{\frac 32, {\rm loc}}(Q_T)$ such that $u^{N}$ and $q^{N}$  satisfy the perturbed Navier-Stokes system in the sense of distributions:
\begin{equation}\label{perturbdirectsystemBp}
\partial_t u^{N}+v\cdot\nabla v-W^{N}\cdot\nabla W^{N}-\Delta u^{N}=-\nabla q^{N},\qquad\mbox{div}\,u^{N}=0
\end{equation}
in $Q_T$.
Furthermore, for any $w\in L_{2}$:
\begin{equation}\label{vweakcontinuityBp}
t\rightarrow \int\limits_{\mathbb R^3} w(x)\cdot u^{N}(x,t)dx
\end{equation}
is a continuous function on  $[0,T]$
and
\begin{equation}\label{initialconditionBp}
\lim_{t\rightarrow 0+} \|u^{N}(\cdot,t)- \widetilde{u}_0^N\|_{L_{2}}=0.
\end{equation}
Moreover, $u$ satisfies the energy inequality:
\begin{equation}\label{venergyineqBp}
\|u^{N}(\cdot,t)\|_{L_{2}}^2+2\int\limits_{0}^t\int\limits_{\mathbb R^3} |\nabla u^{N}(x,t')|^2 dxdt'\leq$$$$\leq 2 \int_{0}^t\int_{\mathbb R^3}(W^{N}\otimes u^{N}):\nabla u^{N}dxdt'+\|\widetilde{u}_{0}^{N}\|_{L_2}^2
\end{equation}
for all $t\in [0,T]$.

Finally, $v$ and $q$ satisfy the local energy inequality.
Namely, for almost every $t\in ]0,T[$ the following inequality holds for all non negative functions $\varphi\in C_{0}^{\infty}(Q_T)$:
\begin{equation}\label{localenergyinequalityBp}
\int\limits_{\mathbb R^3}\varphi(x,t)|v(x,t)|^2dx+2\int\limits_{0}^t\int\limits_{\mathbb R^3}\int\varphi |\nabla v|^2 dxdt^{'}\leq$$$$\leq
\int\limits_{0}^{t}\int\limits_{\mathbb R^3}|v|^2(\partial_{t}\varphi+\Delta\varphi)+v\cdot\nabla\varphi(|v|^2+2q) dxdt^{'}.
\end{equation}

\end{theorem}
\begin{remark}\label{stableestimates}
It can be shown that 
(\ref{mildBesovestimate}), (\ref{venergyineqBp}) and the estimates for $\bar{u}^{N}_{0}$ and $\widetilde{u}^{N}_{0}$ from  Proposition \ref{besovinfinityinterpolationtheorem}, implies the   following estimate for $u^{N}$ in the above Theorem:
\begin{equation}\label{stableestforu}
\sup_{0<t<T}\|u^{N}(\cdot,t)\|_{L_{2}}^2 +\int\limits_0^T\int\limits_{\mathbb{R}^3} |\nabla u^{N}|^2 dxdt'\leq C(T,p, \delta_2, \|u_{0}\|_{\dot{B}^{s_p}_{p,\infty}}, N).
\end{equation}
Recall also that  the Besov mild solution $W^{N}$ from the above Theorem satisfies
\begin{equation}\label{mildest1recap}
\sup_{0<t<T} t^{\frac{p_2}{2}-\frac{\delta_2}{2}} \|W^{N}(\cdot,t)\|_{L_{p_2}} \leq N^{\gamma_{1}}C(p, \|u_0\|_{\dot{B}^{s_{p}}_{p,{\infty}}}).
\end{equation} 
\end{remark}
\begin{subsubsection}{Applications to Regularity Criteria for Weak Leray-Hopf solutions to the Navier-Stokes Equations}
Let us now define the notion of 'weak Leray-Hopf solutions' to the Navier-Stokes system.
 \begin{definition}\label{weakLerayHopf}
Consider $0<S\leq \infty$. Let 
 \begin{equation}\label{initialdataconditionLe}
 u_0\in J(\mathbb{R}^3).
 \end{equation}
 We say that $v$ is a 'weak Leray-Hopf solution' to the Navier-Stokes Cauchy problem in $Q_S:=\mathbb{R}^3 \times ]0,S[$ if it satisfies the following properties:
\begin{equation}
v\in \mathcal{L}(S):= L_{\infty}(0,S; J(\mathbb{R}^3))\cap L_{2}(0,S;\stackrel{\circ} J{^1_2}(\mathbb{R}^3)).
\end{equation} 
Additionally,  for any $w\in L_{2}(\mathbb{R}^3)$:
\begin{equation}\label{vweakcontinuityLe}
t\rightarrow \int\limits_{\mathbb R^3} w(x)\cdot v(x,t)dx
\end{equation}
is a continuous function on  $[0,S]$ (the semi-open interval should be taken if $S=\infty$).
The Navier-Stokes equations are satisfied by $v$ in a weak sense:
\begin{equation}\label{vsatisfiesNSELe}
\int\limits_{0}^{S}\int\limits_{\mathbb{R}^3}(v\cdot \partial_{t} w+v\otimes v:\nabla w-\nabla v:\nabla w) dxdt=0
\end{equation}
for any divergent free test function $$w\in C_{0,0}^{\infty}(Q_{S}):=\{ \varphi\in C_{0}^{\infty}(Q_{S}):\,\,\rm{div}\,\varphi=0\}.$$
The initial condition is satisfied strongly in the $L_{2}(\mathbb{R}^3)$ sense:
\begin{equation}\label{vinitialconditionLe}
\lim_{t\rightarrow 0^{+}} \|v(\cdot,t)-u_0\|_{L_{2}(\mathbb{R}^3)}=0.
\end{equation}

Finally, $v$ satisfies the energy inequality:
\begin{equation}\label{venergyineqLe}
\|v(\cdot,t)\|_{L_{2}(\mathbb{R}^3)}^2+2\int\limits_{0}^t\int\limits_{\mathbb{R}^3} |\nabla v(x,t')|^2 dxdt'\leq \|u_0\|_{L_{2}(\mathbb{R}^3)}^2 
\end{equation}
for all $t\in [0,S]$ (the semi-open interval should be taken if $S=\infty$).
\end{definition}
The corresponding global in time existence result, proven in \cite{Le}, is as follows.
\begin{theorem}
Let $u_0\in J(\mathbb{R}^3)$. Then, there exists at least one weak Leray-Hopf solution on $Q_{\infty}$.
\end{theorem}
There are two big open problems concerning weak Leray-Hopf solutions regarding uniqueness and regularity. Many regularity criteria exist for weak Leray-Hopf solutions.  It was shown by Leray in \cite{Le} that if $v$ is a weak Leray-Hopf solution
 with sufficiently regular initial data and finite blow up time $T$, then  there exists a $C(p)>0$ such that for all $3<p\leq \infty$:
 \begin{equation}\label{Leraynecessary}
 \|v(\cdot,t)\|_{L_p}\geq \frac{c(p)}{(T-t)^{\frac{1}{2}(1-\frac{3}{p})}}.
 \end{equation}
 The case $p=3$ stood long open. It was shown in \cite{ESS2003} that if $v$ is a weak Leray-Hopf solution
 with sufficiently regular initial data and finite blow up time $T$, then necessarily
\begin{equation}\label{ESS2003}
\limsup_{t\uparrow T}\|v(\cdot,t)\|_{L_{3}(\mathbb{R}^3)}.
\end{equation}
The proof in \cite{ESS2003} is by a contradiction argument, involving a rescaling procedure producing a non trivial ancient solution and a Liouville theorem based on backward uniqueness for parabolic equations.
For an alternative approach to this type of regularity criteria, we refer to \cite{kenigkoch}, \cite{gkp13} and \cite{gkp14}.

The criteria (\ref{ESS2003}) was made more precise in \cite{Ser12}, where it was shown that if $T$ is a finite blow up time then necessarily
\begin{equation}\label{sereginsidentity}
\lim_{t\uparrow T}\|v(\cdot,t)\|_{L_{3}(\mathbb{R}^3)}.
\end{equation}
The proof in \cite{Ser12} uses  ideas in \cite{ESS2003}, as well as the fact that the local energy solutions of \cite{LR1} on a fixed time interval, with $L_{3}(\mathbb{R}^3)$ initial data, are continuous with respect to weak convergence in $L_{3}(\mathbb{R}^3)$ of the initial data.  

Unfortunately, it is not known if the notion of local energy solutions in \cite{LR1} carries over to the half space. Consequently, the proof of \cite{Ser12} doesn't apply to the case of weak Leray-Hopf solutions on $\mathbb{R}^3_{+} \times ]0,\infty[$. This was overcome  in \cite{barkerser16blowup}. In particular it was shown that the  $L^{3,q}(\mathbb{R}^3_{+})$ of a weak Leray-Hopf solution $v$ on $\mathbb{R}^3_{+}\times ]0,\infty[$ must tend to infinity with $3\leq q<\infty$. 

 A further refinement has been recently obtained in \cite{dallas}, who showed that if $T$ is a finite blow up time and $3<p<\infty$, then necessarily
$$\lim_{t\uparrow T}\|v(\cdot,t)\|_{\dot{B}_{p,p}^{s_p}(\mathbb{R}^3)}.$$ Note that for $3<p<\infty$
$$L_{3}(\mathbb{R}^3) \hookrightarrow L^{3,p}(\mathbb{R}^3) \hookrightarrow \dot{B}_{p,p}^{s_p}(\mathbb{R}^3)\hookrightarrow \dot{B}_{p,\infty}^{s_p}(\mathbb{R}^3).$$

Theorem \ref{Bpsolutionexistence} and the above Remark (specifically (\ref{stableestforu})-(\ref{mildest1recap})), together with ideas from \cite{dallas} and \cite{barkerser16blowup}, allow us to obtain a new of a regularity criteria for weak Leray-Hopf solutions to the Navier- Stokes equations.
\begin{theorem}\label{criticalBesovnormregularitycriteria}
Let $v$ be a global in time weak Leray-Hopf solution to the Navier-Stokes equations. Assume $0<T<\infty$ is such that for any $0<\epsilon<\frac{T}{2}$
\begin{equation}\label{Tblow up}
v\in L_{\infty}(\epsilon, T-\epsilon, L_{\infty}(\mathbb{R}^3)).
\end{equation}
Additionally assume that there exists an increasing sequence $t_{k}$ tending to $T$ such that
\begin{equation}\label{besovnormbounded}
\sup_{k}\|v(\cdot,t_{k})\|_{\dot{B}^{-1+\frac{3}{p}}_{p,\infty}(\mathbb{R}^3)}=M<\infty.
\end{equation}
Furthermore, assume that for any $\varphi \in C^{\infty}_{0}(\mathbb{R}^3)$
\begin{equation}\label{blowupprofileassumption}
\lim_{\lambda\rightarrow 0}\frac{1}{\lambda^2}\int\limits_{\mathbb{R}^3} v(y,T)\cdot \varphi({y}/{\lambda}) dy=0.
\end{equation}
The assumptions (\ref{Tblow up})-(\ref{blowupprofileassumption})  then imply that for any $0<\epsilon<T$
\begin{equation}\label{Tnotblowup}
v\in L_{\infty}(\epsilon,T; L_{\infty}(\mathbb{R}^3)).
\end{equation}
\end{theorem}
 Once we have Theorem \ref{Bpsolutionexistence} and  Remark \ref{stableestimates} (specifically (\ref{stableestforu})-(\ref{mildest1recap})) the proof the above refined regularity criteria can be completed by verbatim arguments of \cite{dallas}. \footnote{Specifically, Theorem 1.1 of \cite{dallas}.}

It should be noted that removing the assumption (\ref{blowupprofileassumption}) in the above Theorem is a challenging open problem. A positive resolution would provide regularity \footnote{By regularity of $v$ at time $T$, we mean $v\in L_{\infty}(T-\delta,T; L_{\infty}(\mathbb{R}^3))$ for some $0<\delta<T/2$} of $v$ at time $T$ if $$\sup_{x\in\mathbb{R}^3,\,0<t<T}{(|x|+(T-t)^{\frac{1}{2}}})|v(x,t)|<\infty.$$
 For axisymmteric solutions of the Navier-Stokes equations, Type I blow up has been ruled out. See \cite{chenI}-\cite{chenII}, \cite{kochliouville} and \cite{sersve2009}.
\end{subsubsection}

\end{subsection}

\setcounter{equation}{0}
\section{Preliminaries}
\subsection{Relevant Function Spaces}

 We first introduce the frequency cut off operators of the Littlewood-Paley theory. The definitions we use are contained in \cite{bahourichemindanchin}. For a tempered distribution $f$, let  $\mathcal{F}(f)$ denote its Fourier transform.  Let $C$ be the annulus $$\{\xi\in\mathbb{R}^3: 3/4\leq|\xi|\leq 8/3\}.$$
Let $\chi\in C_{0}^{\infty}(B(4/3))$ and $\varphi\in C_{0}^{\infty}(C)$  be such that
\begin{equation}\label{valuesbetween0and1}
\forall \xi\in\mathbb{R}^3,\,\,0\leq \chi(\xi), \varphi(\xi)\leq 1 ,
\end{equation}
\begin{equation}\label{dyadicpartition}
\forall \xi\in\mathbb{R}^3,\,\,\chi(\xi)+\sum_{j\geq 0} \varphi(2^{-j}\xi)=1
\end{equation}
and
\begin{equation}\label{dyadicpartition.1}
\forall \xi\in\mathbb{R}^3\setminus\{0\},\,\,\sum_{j\in\mathbb{Z}} \varphi(2^{-j}\xi)=1.
\end{equation}
For $a$ being a tempered distribution, let us define for $j\in\mathbb{Z}$:
\begin{equation}\label{dyadicblocks}
\dot{\Delta}_j a:=\mathcal{F}^{-1}(\varphi(2^{-j}\xi)\mathcal{F}(a))\,\,\textrm{and}\,\, \dot{S}_j a:=\mathcal{F}^{-1}(\chi(2^{-j}\xi)\mathcal{F}(a)).
\end{equation}
Now we are in a position to define the homogeneous Besov spaces on $\mathbb{R}^3$. Let $s\in\mathbb{R}$ and $(p,q)\in [1,\infty]\times [1,\infty]$. Then $\dot{B}^{s}_{p,q}(\mathbb{R}^3)$ is the subspace of tempered distributions such that
\begin{equation}\label{Besovdef1}
\lim_{j\rightarrow-\infty} \|\dot{S}_{j} u\|_{L_{\infty}(\mathbb{R}^3)}=0,
\end{equation}
\begin{equation}\label{Besovdef2}
\|u\|_{\dot{B}^{s}_{p,q}(\mathbb{R}^3)}:= \Big(\sum_{j\in\mathbb{Z}}2^{jsq}\|\dot{\Delta}_{j} u\|_{L_p(\mathbb{R}^3)}^{q}\Big)^{\frac{1}{q}}.
\end{equation}
\begin{remark}\label{besovremark1}
 This definition provides a Banach space if $s<\frac{3}{p}$, see \cite{bahourichemindanchin}.
\end{remark}
\begin{remark}\label{besovremark2}
It is known that if $1\leq q_{1}\leq q_{2}\leq\infty$, $1\leq p_1\leq p_2\leq\infty$ and $s\in\mathbb{R}$, then
$$\dot{B}^{s}_{p_1,q_{1}}(\mathbb{R}^3)\hookrightarrow\dot{B}^{s-3(\frac{1}{p_1}-\frac{1}{p_2})}_{p_2,q_{2}}(\mathbb{R}^3).$$
\end{remark}
\begin{remark}\label{besovremark3}
It is known that for $s=-2s_{1}<0$ and $p,q\in [1,\infty]$, the norm can be characterised by the heat flow. Namely there exists a $C>1$ such that for all $u\in\dot{B}^{-2s_{1}}_{p,q}(\mathbb{R}^3)$:
$$C^{-1}\|u\|_{\dot{B}^{-2s_{1}}_{p,q}(\mathbb{R}^3)}\leq \|\|t^{s_1} S(t)u\|_{L_{p}(\mathbb{R}^3)}\|_{L_{q}(\frac{dt}{t})}\leq C\|u\|_{\dot{B}^{-2s_{1}}_{p,q}(\mathbb{R}^3)}.$$  Here, $$S(t)u(x):=\Gamma(\cdot,t)\star u$$
where $\Gamma(x,t)$ is the kernel for the heat flow in $\mathbb{R}^3$.
\end{remark}
We will also need the following Proposition, whose statement and proof can be found in \cite{bahourichemindanchin} (Proposition 2.22 there) for example. In the Proposition below we use the notation
\begin{equation}\label{Sh}
\mathcal{S}_{h}^{'}:=\{ \textrm{ tempered\,\,distributions}\,\, u\textrm{\,\,\,such\,\,that\,\,} \lim_{j\rightarrow -\infty}\|S_{j}u\|_{L_{\infty}(\mathbb{R}^3)}=0\}.
\end{equation}
\begin{pro}\label{interpolativeinequalitybahourichemindanchin}
A constant $C$ exists with the following properties. If $s_{1}$ and $s_{2}$ are real numbers such that $s_{1}<s_{2}$ and $\theta\in ]0,1[$, then we have, for any $p\in [1,\infty]$ and any $u\in \mathcal{S}_{h}^{'}$,
\begin{equation}\label{interpolationactual}
\|u\|_{\dot{B}_{p,1}^{\theta s_{1}+(1-\theta)s_{2}}(\mathbb{R}^3)}\leq \frac{C}{s_2-s_1}\Big(\frac{1}{\theta}+\frac{1}{1-\theta}\Big)\|u\|_{\dot{B}_{p,\infty}^{s_1}(\mathbb{R}^3)}^{\theta}\|u\|_{\dot{B}_{p,\infty}^{s_2}(\mathbb{R}^3)}^{1-\theta}.
\end{equation}
\end{pro}
\begin{pro}\label{weak*approx}
Let $u_{0}\in \B$ be divergence free, in the sense of  tempered distributions. 
Then there exists a  weakly divergence free sequence $u^{(k)}_{0}\in {L_3(\mathbb{R}^3)}$ such that 
$$u^{(k)}_{0}\stackrel{*}{\rightharpoonup} u_0$$ in $\B$.
\end{pro}
\begin{proof}
Next, it is well known that for any $ u_0\in \B(\mathbb{R}^3)$ we have that
$u_{0,<|k|}:=\sum_{j=-k}^{k} \dot{\Delta}_{j} u_0$ converges to $u_0$ in the sense of tempered distributions. 
Furthermore, we have  that $\dot{\Delta}_{j}\dot{\Delta}_{j'} u_0=0$ if $|j-j'|>1.$ Thus,
\begin{equation}\label{boundedinB}
\|u_{0,<|k|}\|_{\B}\leq C \|u_0\|_{\B}
\end{equation} 
and \begin{equation}\label{othogonality}
\dot{\Delta}_{N}u_{0,<|k|}=0\,\,\,\,\textrm{if}\,\,\,\,|N|>k+1.
\end{equation}
It then follows from \cite{bahourichemindanchin} \footnote{ Specifically, Remark 2.2 pg 69 of \cite{bahourichemindanchin}. } 
that there exists a Schwartz function $g^{(k)}$, whose Fourier transform is supported away from the origin, such that
\begin{equation}\label{u0truncated}
\|g^{(k)}-u_{0,<|k|}\|_{\B}<\frac{1}{k}.
\end{equation}
Define $u_{0}^{(k)}$ to be the Leray Projector $\mathbb{P}$ applied to $g^{(k)}$, which is continuous on $\B(\mathbb{R}^3)$. Obviously since $g^{(k)}$ is Schwartz, we have that $u_{0}^{(k)}$  is a weakly divergence free function in $L_{3}(\mathbb{R}^3)$. Since $u_0$ is divergence free in the sense of tempered distributions, we have that
$$\mathbb{P} u_{0,<|k|}=u_{0,<|k|}.$$
Thus, 
\begin{equation}\label{u_0approximation}
\|u^{(k)}_0-u_{0,<|k|}\|_{\B}=\|\mathbb{P}(g^{(k)}_0-u_{0,<|k|})\|_{\B}<\frac{C}{k}.
\end{equation}
So for any Schwartz function $\varphi$ we have
$$|<u^{(k)}_0-u_{0}, \varphi>|\leq |<u^{(k)}_0-u_{0<|k|}, \varphi>|+
|<u_0-u_{0<|k|}, \varphi>|\leq $$$$\leq \frac{C}{k} \|\varphi\|_{\dot{B}^{\frac{1}{4}}_{\frac{4}{3},1}}+|<u_0-u_{0<|k|}, \varphi>|.$$
Thus, $u_0^{(k)}$ satisfies the desired properties of the Proposition. 

\end{proof}

\begin{pro}\label{semigroupweakL3}
Let $u_0\in\B(\mathbb{R}^3)$. Then we have 
\begin{equation}\label{L3inftybdd}
\sup_{0<t<\infty} t^{\frac{1}{8}}\|S(t)u_{0}\|_{\B}\leq C\|u_{0}\|_{\B}.
\end{equation}
Moreover for $4\leq r<\infty$, $m,k\in \mathbb{N}$:
\begin{equation}\label{semigroupsolonnikovest}
\|\partial^m_t\nabla^k S(t)u_0\|_{L_{r}}\leq \frac{C(m,k,r)\|u_{0}\|_{\B}}{t^{m+\frac k 2+{\frac 1 2}(1-\frac{3}{r})}}.
\end{equation}

\end{pro}
\setcounter{equation}{0}
\section{Decomposition of Homogeneous Besov Spaces}
Before proving Proposition \ref{besovinfinityinterpolationtheorem}, we take note of a useful Lemma presented in \cite{bahourichemindanchin} (specifically, Lemma 2.23 and Remark 2.24 in \cite{bahourichemindanchin}).
\begin{lemma}\label{bahouricheminbook}
Let $C^{'}$ be an annulus and let $(u^{(j)})_{j\in\mathbb{Z}}$ be a sequence of functions such that
\begin{equation}\label{conditionsupport}
\textrm{Supp}\, \mathcal{F}(u^{(j)})\subset 2^{j}C^{'} 
\end{equation}
and
\begin{equation}\label{conditionseriesbound}
\Big(\sum_{j\in\mathbb{Z}}2^{jsr}\|u^{(j)}\|_{L_p}^{r}\Big)^{\frac{1}{r}}<\infty.
\end{equation}
Moreover, assume in addition that 
\begin{equation}\label{indicescondition}
s<\frac{3}{p}.
\end{equation}
Then the following holds true.
The series $$\sum_{j\in\mathbb{Z}} u^{(j)}$$ converges (in the sense of tempered distributions) to some $u\in \dot{B}^{s}_{p,r}(\mathbb{R}^3)$, which satisfies the following estimate:
\begin{equation}\label{besovboundlimitfunction}
\|u\|_{\dot{B}^{s}_{p,r}}\leq C_s \Big(\sum_{j\in\mathbb{Z}}2^{jsr}\|u^{(j)}\|_{L_p}^{r}\Big)^{\frac{1}{r}}.
\end{equation}

\end{lemma}
Let us state a useful Lemma, proven in \cite{barker2017}, regarding decomposition of homogeneous Besov spaces.
\begin{pro}\label{Decompgeneral}
For $i=1,2,3$ let $p_{i}\in ]1,\infty[$, $s_i\in \mathbb{R}$ and $\theta\in ]0,1[$ be such that $s_1<s_0<s_2$ and $p_2<p_0<p_1$. In addition, assume the following relations hold:
\begin{equation}\label{sinterpolationrelation}
s_1(1-\theta)+\theta s_2=s_0,
\end{equation}
\begin{equation}\label{pinterpolationrelation}
\frac{1-\theta}{p_1}+\frac{\theta}{p_2}=\frac{1}{p_0},
\end{equation}
\begin{equation}\label{besovbanachcondition}
{s_i}<\frac{3}{p_i}.
\end{equation}
Suppose that $u_0\in \dot{B}^{{s_{0}}}_{p_0,p_0}(\mathbb{R}^3).$
Then for all $\epsilon>0$ there exists $u^{1,\epsilon}\in \dot{B}^{s_{1}}_{p_1,p_1}(\mathbb{R}^3)$, $u^{2,\epsilon}\in \dot{B}^{s_{2}}_{p_2,p_2}(\mathbb{R}^3)$ such that 
\begin{equation}\label{udecompgeneral}
u= u^{1,\epsilon}+u^{2,\epsilon},
\end{equation}
\begin{equation}\label{u_1estgeneral}
\|u^{1,\epsilon}\|_{\dot{B}^{s_{1}}_{p_1,p_1}}^{p_1}\leq \epsilon^{p_1-p_0} \|u_0\|_{\dot{B}^{s_{0}}_{p_0,p_0}}^{p_0},
\end{equation}
\begin{equation}\label{u_2estgeneral}
\|u^{2,\epsilon}\|_{\dot{B}^{s_{2}}_{p_2,p_2}}^{p_2}\leq C( p_0,p_1,p_2, \|\mathcal{F}^{-1}\varphi\|_{L_1})\epsilon^{p_2-p_0} \|u_0\|_{\dot{B}^{s_{0}}_{p_0,p_0}}^{p_0}
.\end{equation}
\end{pro}
In order to prove Proposition \ref{besovinfinityinterpolationtheorem}, we  must first state and prove two Lemmas. Here is the first.
\begin{lemma}\label{Decompgeneralinfinity}
Let  $3<p<\infty$ and suppose $u_0 \in \dot{B}^{s_p}_{p,\infty}(\mathbb{R}^3)$.
Then  there exists $p<p_0<\infty$ and $0<\delta<-s_{p_0}$ such that  for any $N>0$ there exists functions 
$\bar{u}^{1,N}\in \dot{B}^{s_{p_0}+\delta}_{p_0,\infty}(\mathbb{R}^3)\cap \dot{B}^{s_{p}}_{p,\infty}(\mathbb{R}^3) $ and $\bar{u}^{2,N}\in \dot{B}^{0}_{2,\infty}(\mathbb{R}^3) \cap \dot{B}^{s_{p}}_{p,\infty}(\mathbb{R}^3) $ with
 
\begin{equation}\label{udecomp1infinity}
u_0= {u}^{1,N}+{u}^{2,N},
\end{equation}
\begin{equation}\label{baru_1estinfinity}
\|{u}^{1,N}\|_{\dot{B}^{s_{p_0}+\delta}_{p_0,\infty}}^{p_0}\leq N^{p_0-p} \|u_0\|_{\dot{B}^{s_{p}}_{p,\infty}}^{p},
\end{equation}
\begin{equation}\label{baru_2estinfinity}
\|{u}^{2,N}\|_{\dot{B}^{0}_{2,\infty}}^2\leq C(p,p_0,\|\mathcal{F}^{-1}\varphi\|_{L_1}) N^{2-p}\|u_0\|_{\dot{B}^{s_{p}}_{p,\infty}}^p
,
\end{equation} 
\begin{equation}\label{baru_1est.1infinity}
\|{u}^{1,N}\|_{\dot{B}^{s_{p}}_{p,\infty}}\leq  C(\|\mathcal{F}^{-1}\varphi\|_{L_1})\|u_0\|_{\dot{B}^{s_{p}}_{p,\infty}}
\end{equation}
and
\begin{equation}\label{baru_2est.1infinity}
\|{u}^{2,N}\|_{\dot{B}^{s_{p}}_{p,\infty}}\leq C(\|\mathcal{F}^{-1}\varphi\|_{L_1})\|u_0\|_{\dot{B}^{s_{p}}_{p,\infty}}.
\end{equation}
\end{lemma}
\begin{proof}
If $p<p_0<\infty$, there exists a $\theta\in ]0,1[$ such that
\begin{equation}\label{pporelation}
\frac{1-\theta}{p_0}+\frac{\theta}{2}=\frac{1}{p}.
\end{equation}
If we define
\begin{equation}\label{deltadefinfinity}
\delta:=\frac{\theta}{2(1-\theta)}>0
\end{equation}
we see that
\begin{equation}\label{criticalindicesrelation}
(s_{p_0}+\delta)(1-\theta)=s_{p}.
\end{equation}
Denote, $$f^{(j)}:= \dot{\Delta}_{j} u,$$ $$f^{(j)M}_{-}:=f^{(j)}\chi_{|f^{(j)}|\leq M} $$ and $$f^{(j)M}_{+}:=f^{(j)}(1-\chi_{|f^{(j)}|\leq M}). $$ 
It is easily verified that the following holds:
$$ \|f^{(j)M}_{-}\|_{L_{p_0}}^{p_0}\leq M^{p_0-p}\|f^{(j)}\|_{L_{p}}^{p},$$
$$ \|f^{(j)M}_{-}\|_{L_{p}}\leq\|f^{(j)}\|_{L_{p}},$$
$$\|f^{(j)M}_{+}\|_{L_{2}}^{2}\leq M^{2-p}\|f^{(j)}\|_{L_{p}}^{p}$$
and 
$$\|f^{(j)M}_{+}\|_{L_{p}}\leq \|f^{(j)}\|_{L_{p}}.$$
Thus, we may write
\begin{equation}\label{truncationest1infinity}
2^{p_0(s_{p_0}+\delta) j}\|f^{(j)M}_{-}\|_{L_{p_0}}^{p_0}\leq M^{p_0-p}2^{(p_0(s_{p_0}+\delta)-p s_p)j} 2^{p s_p j}\|f^{(j)}\|_{L_{p}}^{p}
\end{equation}
\begin{equation}\label{truncationest2infinity}
\|f^{(j)M}_{+}\|_{L_{2}}^{2}\leq M^{2-p} 2^{-ps_pj} 2^{ps_pj}\|f^{(j)}\|_{L_{p}}^{p}.
\end{equation}
With (\ref{truncationest1infinity}) in mind, we define 
$$M({j,N,p,p_0,\delta}):= N 2^{\frac{(p s_p-p_0(s_{p_0}+\delta))j}{p_0-p}}.$$
For the sake of brevity we will write $M({j},N)$.
Using the relations of the Besov indices given by (\ref{pporelation})-(\ref{deltadefinfinity}), we can infer that
$$M(j,N)^{2-p} 2^{-ps_pj} = N^{p_2-p_0}.$$
 The crucial point being that this is independent of $j$.
 Thus, we infer
 \begin{equation}\label{truncationest1.1infinity}
 2^{p_0(s_{p_0}+\delta) j}\|f^{(j)M(j,N)}_{-}\|_{L_{p_0}}^{p_0}\leq N^{p_0-p} 2^{p s_p j}\|f^{(j)}\|_{L_{p}}^{p},
 \end{equation}
 \begin{equation}\label{truncationest1.2infinity}
  2^{s_{p} j}\|f^{(j)M(j,N)}_{-}\|_{L_{p}}\leq 2^{ s_p j}\|f^{(j)}\|_{L_{p}},
 \end{equation}
\begin{equation}\label{truncationest2.1infinity}
 \|f^{(j)M(j,N)}_{+}\|_{L_{2}}^{2}\leq N^{2-p} 2^{p s_pj}\|f^{(j)}\|_{L_{p}}^{p}
 \end{equation}
 and
 \begin{equation}\label{truncationest2.2infinity}
  2^{s_{p} j}\|f^{(j)M(j,N)}_{+}\|_{L_{p}}\leq 2^{ s_p j}\|f^{(j)}\|_{L_{p}}.
 \end{equation}
Next, it is well known that for any $ u\in \dot{B}^{s_p}_{p,\infty}(\mathbb{R}^3)$ we have that
$\sum_{j=-m}^{m} \dot{\Delta}_{j} u$ converges to $u$ in the sense of tempered distributions.
Furthermore, we have  that $\dot{\Delta}_{j}\dot{\Delta}_{j'} u=0$ if $|j-j'|>1.$ Combing these two facts allows us to observe that
\begin{equation}\label{smoothingtruncationsinfinity}
\dot{\Delta}_{j} u= \sum_{|m-j|\leq 1} \dot{\Delta}_{m} f^{(j)}= \sum_{|m-j|\leq 1}\dot{\Delta}_{m} f_{-}^{(j)M(j,N)}+\sum_{|m-j|\leq 1}\dot{\Delta}_{m} f_{+}^{(j)M(j,N)}.
\end{equation}
Define 
\begin{equation}\label{decomp1eachpieceinfinity}
u^{1,N}_{j}:= \sum_{|m-j|\leq 1}\dot{\Delta}_{m} f_{-}^{(j)M(j,N)},
\end{equation}
\begin{equation}\label{decomp2eachpieceinfinity}
u^{2,N}_{j}:= \sum_{|m-j|\leq 1}\dot{\Delta}_{m} f_{+}^{(j)M(j,N)}
\end{equation}
It is clear, that \begin{equation}\label{fouriersupportinfinity}
\textrm{Supp}\,\mathcal{F}(u^{1,N}_{j}), \textrm{Supp}\,\mathcal{F}(u^{2,N}_{j})\subset 2^{j}C^{'}.
\end{equation}
Here, $C'$ is the annulus defined by $C':=\{\xi\in\mathbb{R}^3: 3/8\leq |\xi|\leq 16/3\}.$
Using, (\ref{truncationest1.1infinity})-(\ref{truncationest2.2infinity}) we can obtain
the following estimates:
\begin{equation}\label{decomp1estinfinity}
2^{p_0(s_{p_0}+\delta) j}\|u^{1,N}_{j}\|_{L_{p_0}}^{p_0}\leq \lambda_{1}(p_0, \|\mathcal{F}^{-1} \varphi\|_{L_1}) 2^{p_0(s_{p_0}+\delta) j}\|f^{(j)M(j,N)}_{-}\|_{L_{p_0}}^{p_0}\leq$$$$\leq \lambda_{1}(p_0, \|\mathcal{F}^{-1} \varphi\|_{L_1}) N^{p_0-p} 2^{p s_p j}\|f^{(j)}\|_{L_{p}}^{p},
\end{equation}
\begin{equation}\label{decomp1.1estinfinity}
2^{s_{p} j}\|u^{1,N}_{j}\|_{L_{p}}\leq \lambda_{2}( \|\mathcal{F}^{-1} \varphi\|_{L_1}) 2^{s_{p} j}\|f^{(j)M(j,N)}_{-}\|_{L_{p}}\leq$$$$\leq \lambda_{2}( \|\mathcal{F}^{-1} \varphi\|_{L_1}) 2^{ s_p j}\|f^{(j)}\|_{L_{p}},
\end{equation}
\begin{equation}\label{decomp2estinfinity}
\|u^{2,N}_{j}\|_{L_{2}}^{2}\leq \lambda_{3}( \|\mathcal{F}^{-1} \varphi\|_{L_1}) \|f^{(j)M(j,N)}_{+}\|_{L_{2}}^{2}\leq$$$$\leq \lambda_{3}( \|\mathcal{F}^{-1} \varphi\|_{L_1})N^{2-p} 2^{p s_p j}\|f^{(j)}\|_{L_{p}}^{p}
\end{equation}
and
\begin{equation}\label{decomp2.1estinfinity}
2^{s_{p} j}\|u^{2,N}_{j}\|_{L_{p}}\leq \lambda_{4}( \|\mathcal{F}^{-1} \varphi\|_{L_1}) 2^{s_{p} j}\|f^{(j)M(j,N)}_{+}\|_{L_{p}}\leq$$$$\leq \lambda_{4}( \|\mathcal{F}^{-1} \varphi\|_{L_1}) 2^{ s_p j}\|f^{(j)}\|_{L_{p}}.
\end{equation}
It is then the case that (\ref{fouriersupportinfinity})-(\ref{decomp2.1estinfinity}) allow us to apply the results of Lemma \ref{bahouricheminbook}. This allows us to achieve the desired decomposition with the choice
$$u^{1,N}=\sum_{j\in\mathbb{Z}}u^{1,N}_{j},$$
$$u^{2,N}=\sum_{j\in\mathbb{Z}}u^{2,N}_{j}.$$

\end{proof}
\begin{lemma}\label{Decomp}
Fix $2<\alpha< 3.$
\begin{itemize}
\item For $2<\alpha< 3$, take $p$ such that $3 <p< \frac{\alpha}{3-\alpha}$.

\end{itemize}
For $p$ and $\alpha$ satisfying these conditions, suppose that 
\begin{equation}
u_0\in \dot{B}^{s_{p,\alpha}}_{p,p}(\mathbb{R}^3)\cap \dot{B}^{s_{p}}_{p,\infty}(\mathbb{R}^3)
\end{equation}

  Then the above assumptions imply that there exists $p<p_1<\infty$ and \\$0<\delta_{1}<-s_{p_1}$ such that  for any $\epsilon>0$ there exists functions 
${U}^{1,\epsilon}\in \dot{B}^{s_{p_1}+\delta_{1}}_{p_1,p_1}(\mathbb{R}^3)\cap \dot{B}^{s_p}_{p,\infty}(\mathbb{R}^3)$ and ${U}^{2,\epsilon}\in L_2(\mathbb{R}^3)\cap \dot{B}^{s_p}_{p,\infty}(\mathbb{R}^3)$ with
 
\begin{equation}\label{udecomp1}
u_0= {U}^{1,\epsilon}+{U}^{2,\epsilon},
\end{equation}
\begin{equation}\label{baru_1est}
\|{U}^{1,\epsilon}\|_{\dot{B}^{s_{p_1}+\delta_1}_{p_1,p_1}}^{p_1}\leq \epsilon^{p_1-p} \|u_0\|_{\dot{B}^{s_{p,\alpha}}_{p,p}}^{p},
\end{equation}
\begin{equation}\label{baru_1est.2}
\|{U}^{2,\epsilon}\|_{\dot{B}^{s_{p}}_{p,\infty}}\leq  C(\|\mathcal{F}^{-1}\varphi\|_{L_1})\|u_0\|_{\dot{B}^{s_{p}}_{p,\infty}},
\end{equation}
\begin{equation}\label{baru_2est}
\|{U}^{2,\epsilon}\|_{L_2}^2\leq C(p,p_1,\|\mathcal{F}^{-1}\varphi\|_{L_1}) \epsilon^{2-p}\|u_0\|_{\dot{B}^{s_{p,\alpha}}_{p,p}}^p
\end{equation} 
and
\begin{equation}\label{baru_2est.2}
\|{U}^{2,\epsilon}\|_{\dot{B}^{s_{p}}_{p,\infty}}\leq  C(\|\mathcal{F}^{-1}\varphi\|_{L_1})\|u_0\|_{\dot{B}^{s_{p}}_{p,\infty}}.
\end{equation}

\end{lemma}
\begin{proof}

Under this condition, we can find $p<p_{1}<\infty$ such that
\begin{equation}\label{condition}
\theta:= \frac{\frac{1}{p}-\frac{1}{p_1}}{\frac{1}{2}-\frac{1}{p_1}}>\frac{6}{\alpha}-2.
\end{equation}
Clearly, $0<\theta<1$ and moreover
\begin{equation}\label{summabilityindicerelation}
\frac{1-\theta}{p_1}+\frac{\theta}{2}=\frac{1}{p}.
\end{equation}
Define
\begin{equation}\label{deltadef}
\delta_{1}:=\frac{1-\frac{3}{\alpha}+\frac{\theta}{2}}{1-\theta}.
\end{equation}
From (\ref{condition}), we see that $\delta_1>0$.
One can also see we have the following relation: 
\begin{equation}\label{regularityindicerelation}
(1-\theta)(s_{p_1}+\delta_{1})=s_{p,\alpha}.
\end{equation}

The above relations allow us  to apply Proposition \ref{Decompgeneral} to obtain the following decomposition:
(we note that $\dot{B}^{0}_{2,2}(\mathbb{R}^3)$ coincides with $L_2(\mathbb{R}^3)$ with equivalent norms)
\begin{equation}\label{udecomp}
u_0= {U}^{1,\epsilon}+{U}^{2,\epsilon},
\end{equation}
\begin{equation}\label{u_1est}
\|{U}^{1,\epsilon}\|_{\dot{B}^{s_{p_1}+\delta_1}_{p_1,p_1}}^{p_1}\leq \epsilon^{p_1-p} \|u_0\|_{\dot{B}^{s_{p,\alpha}}_{p,p}}^p,
\end{equation}
\begin{equation}\label{u_2est}
\|{U}^{2,\epsilon}\|_{L_2}^2\leq C(p,p_1,\|\mathcal{F}^{-1}\varphi\|_{L_1})\epsilon^{2-p} \|u_0\|_{\dot{B}^{s_{p,\alpha}}_{p,p}}^p
.
\end{equation} 
For $j\in\mathbb{Z}$ and $m\in\mathbb{Z}$, it can be seen that  \begin{equation}\label{besovpersistency1}
\|\dot{\Delta}_{m}\left( (\dot{\Delta}_{j}u_{0})\chi_{|\dot{\Delta}_{j}u_{0}|\leq N(j,\epsilon)}\right)\|_{L_{p}}\leq C(\|\mathcal{F}^{-1} \varphi\|_{L_1})) \|\dot{\Delta}_{j} u_0\|_{L_p}.
\end{equation}
and
\begin{equation}\label{besovpersistency2}
\|\dot{\Delta}_{m}\left( (\dot{\Delta}_{j}u_{0})\chi_{|\dot{\Delta}_{j}u_{0}|\geq N(j,\epsilon)}\right)\|_{L_{p}}\leq C(\|\mathcal{F}^{-1} \varphi\|_{L_1})) \|\dot{\Delta}_{j} u_0\|_{L_p}.
\end{equation}
From \cite{barker2017}, we see that the definitions of $U^{1,\epsilon}$ and $U^{2,\epsilon}$ used in Proposition \ref{Decompgeneral} are of the following form:
$$U^{1,\epsilon}:= \sum_{j}\sum_{|m-j|\leq 1}\dot{\Delta}_{m}\left( (\dot{\Delta}_{j}u_{0})\chi_{|\dot{\Delta}_{j}u_{0}|\leq N(j,\epsilon)}\right)$$
and
$$U^{2,\epsilon}:= \sum_{j}\sum_{|m-j|\leq 1}\dot{\Delta}_{m}\left( (\dot{\Delta}_{j}u_{0})\chi_{|\dot{\Delta}_{j}u_{0}|\geq N(j,\epsilon)}\right)$$
Using this,  along with (\ref{besovpersistency1})-(\ref{besovpersistency2}), we can infer that
$$\|{U}^{1,\epsilon}\|_{\dot{B}^{s_{p}}_{p,\infty}}\leq   C(\|\mathcal{F}^{-1} \varphi\|_{L_1})\|u_0\|_{\dot{B}^{s_{p}}_{p,\infty}}$$
and
$$\|{U}^{2,\epsilon}\|_{\dot{B}^{s_{p}}_{p,\infty}}\leq   C(\|\mathcal{F}^{-1} \varphi\|_{L_1})\|u_0\|_{\dot{B}^{s_{p}}_{p,\infty}}.$$


\end{proof}

\begin{itemize}
\item[]\textbf{Proof of Proposition \ref{besovinfinityinterpolationtheorem}}
\end{itemize}
\begin{proof}
Applying Lemma \ref{Decompgeneralinfinity}, we see that there exists $p<p_0<\infty$ and $0<\delta<-s_{p_0}$ such that  for any $N>0$ there exists functions 
${u}^{1,N}\in \dot{B}^{s_{p_0}+\delta}_{p_0,\infty}(\mathbb{R}^3)\cap \dot{B}^{s_{p}}_{p,\infty}(\mathbb{R}^3) $ and ${u}^{2,N}\in \dot{B}^{0}_{2,\infty}(\mathbb{R}^3) \cap \dot{B}^{s_{p}}_{p,\infty}(\mathbb{R}^3) $ with
 
\begin{equation}\label{udecomp1infinityrecall}
g= {u}^{1,N}+{u}^{2,N},
\end{equation}
\begin{equation}\label{baru_1estinfinityrecall}
\|{u}^{1,N}\|_{\dot{B}^{s_{p_0}+\delta}_{p_0,\infty}}^{p_0}\leq N^{p_0-p} \|g\|_{\dot{B}^{s_{p}}_{p,\infty}}^{p},
\end{equation}
\begin{equation}\label{baru_1est.1infinityrecall}
\|{u}^{1,N}\|_{\dot{B}^{s_{p}}_{p,\infty}}\leq  C(\|\mathcal{F}^{-1}\varphi\|_{L_1})\|g\|_{\dot{B}^{s_{p}}_{p,\infty}},
\end{equation}
\begin{equation}\label{baru_2estinfinityrecall}
\|{u}^{2,N}\|_{\dot{B}^{0}_{2,\infty}}^2\leq C(p,p_0,\|\mathcal{F}^{-1}\varphi\|_{L_1}) N^{2-p}\|g\|_{\dot{B}^{s_{p}}_{p,\infty}}^p
\end{equation} 
and
\begin{equation}\label{baru_2est.1infinityrecall}
\|{u}^{2,N}\|_{\dot{B}^{s_{p}}_{p,\infty}}\leq C(\|\mathcal{F}^{-1}\varphi\|_{L_1})\|g\|_{\dot{B}^{s_{p}}_{p,\infty}}.
\end{equation}
Since $3<p<\infty$, it is clear that there exists an $\alpha:=\alpha(p)$ such that $2<\alpha<3$ and
\begin{equation}\label{pcondition}
3<p<\frac{\alpha}{3-\alpha}.
\end{equation}
With this $p$ and $\alpha$, we may apply Proposition \ref{interpolativeinequalitybahourichemindanchin} with
$s_{1}= -\frac{3}{2}+\frac{3}{p},$
$s_{2}= -1+\frac{3}{p}$
and $\theta=6\Big(\frac{1}{\alpha}-\frac{1}{3}\Big).$ In particular this gives for any $f\in \mathcal{S}^{'}_{h}$:
\begin{equation}\label{interpolativeinequality}
\|f\|_{\dot{B}^{s_{p,\alpha}}_{p,1}}\leq c(p,\alpha)\|f\|_{\dot{B}^{-\frac{3}{2}+\frac{3}{p}}_{p,\infty}}^{6(\frac{1}{\alpha}-\frac{1}{3})}\|f\|_{\dot{B}^{-1+\frac{3}{p}}_{p,\infty}}^{6(\frac{1}{2}-\frac{1}{\alpha})}.
\end{equation}
From Remark \ref{besovremark2}, we see that $\dot{B}^{s_{p,\alpha}}_{p,1}(\mathbb{R}^3)\hookrightarrow \dot{B}^{s_{p,\alpha}}_{p,p}(\mathbb{R}^3) $ and
$\dot{B}^{0}_{2,\infty}(\mathbb{R}^3)\hookrightarrow\dot{B}^{-\frac{3}{2}+\frac{3}{p}}_{p,\infty}(\mathbb{R}^3).$
Thus, we have the inclusion 
\begin{equation}\label{besovinclusionforestnearinitialtime}
\dot{B}^{s_{p}}_{p,\infty}(\mathbb{R}^3)\cap \dot{B}^{0}_{2,\infty}(\mathbb{R}^3)\subset\dot{B}^{s_{p,\alpha}}_{p,p}(\mathbb{R}^3)\cap \dot{B}^{s_{p}}_{p,\infty}(\mathbb{R}^3) . 
\end{equation}
Now (\ref{interpolativeinequality})-(\ref{besovinclusionforestnearinitialtime}), together with (\ref{baru_2estinfinityrecall})-(\ref{baru_2est.1infinityrecall}), imply that $u^{2,N} \in \dot{B}^{s_{p,\alpha}}_{p,p}(\mathbb{R}^3)$ and there exists $\beta({p})>0$ such that
\begin{equation}\label{u2epsilonestinterpolativeinequality}
\|u^{2,N}\|_{\dot{B}^{s_{p,\alpha}}_{p,p}} \leq N^{-\beta(p)} C(p,p_0,\|\mathcal{F}^{-1}\varphi\|_{L_1},\|g\|_{\dot{B}^{s_{p}}_{p,\infty}}).
\end{equation}
Noting (\ref{pcondition}), along with (\ref{baru_2est.1infinityrecall}) and (\ref{u2epsilonestinterpolativeinequality}), we may now apply Lemma \ref{Decomp}. In particular, there exists $p<p_1<\infty$ and $0<\delta_{1}<-s_{p_1}$ such that  for any $\lambda>0$ there exists functions 
${U}^{1,\lambda}\in \dot{B}^{s_{p_1}+\delta_{1}}_{p_1,p_1}(\mathbb{R}^3)\cap \dot{B}^{s_{p}}_{p,\infty}(\mathbb{R}^3)$ and ${U}^{2,\lambda}\in L_2(\mathbb{R}^3)\cap \dot{B}^{s_{p}}_{p,\infty}(\mathbb{R}^3)$ with
 
\begin{equation}\label{udecomp1lambda}
u^{2,N}= {U}^{1,\lambda}+{U}^{2,\lambda},
\end{equation}
\begin{equation}\label{baru_1estlambda}
\|{U}^{1,\lambda}\|_{\dot{B}^{s_{p_1}+\delta_1}_{p_1,p_1}}^{p_1}\leq \lambda^{p_1-p} \|u^{2,N}\|_{\dot{B}^{s_{p,\alpha}}_{p,p}}^{p}\leq$$$$\leq \lambda^{p_1-p}N^{-\beta(p)p} C(p,p_0,\|\mathcal{F}^{-1}\varphi\|_{L_1},\|g\|_{\dot{B}^{s_{p}}_{p,\infty}}),
\end{equation}

\begin{equation}\label{baru_2est.21lambda}
\|{U}^{1,\lambda}\|_{\dot{B}^{s_{p}}_{p,\infty}}\leq  C(\|\mathcal{F}^{-1}\varphi\|_{L_1})\|u^{2,N}\|_{\dot{B}^{s_{p}}_{p,\infty}}$$$$\leq  C(\|\mathcal{F}^{-1}\varphi\|_{L_1})\|g\|_{\dot{B}^{s_{p}}_{p,\infty}},
\end{equation}
\begin{equation}\label{baru_2estlambda}
\|{U}^{2,\lambda}\|_{L_2}^2\leq C(p,p_1,\|\mathcal{F}^{-1}\varphi\|_{L_1}) \lambda^{2-p}\|u^{2,N}\|_{\dot{B}^{s_{p,\alpha}}_{p,p}}^p\leq$$$$\leq  \lambda^{2-p} N^{-\beta(p)p}C(p,p_0,p_1,\|\mathcal{F}^{-1}\varphi\|_{L_1},\|g\|_{\dot{B}^{s_{p}}_{p,\infty}})  
\end{equation} 
and 
\begin{equation}\label{baru_2est.2lambda}
\|{U}^{2,\lambda}\|_{\dot{B}^{s_{p}}_{p,\infty}}\leq  C(\|\mathcal{F}^{-1}\varphi\|_{L_1})\|u^{2,N}\|_{\dot{B}^{s_{p}}_{p,\infty}}$$$$\leq  C(\|\mathcal{F}^{-1}\varphi\|_{L_1})\|g\|_{\dot{B}^{s_{p}}_{p,\infty}}.
\end{equation}
Taking $\lambda=N^{\kappa}$ 
gives that $u_{0}= u^{1,N}+ U^{1,N^{\kappa}}+ U^{2, N^{\kappa}}$ with $u^{1,N} \in \dot{B}^{s_{p_0}+\delta}_{p_0,\infty}(\mathbb{R}^3)\cap \dot{B}^{s_{p}}_{p,\infty}(\mathbb{R}^3)$, $U^{1,N^{\kappa}}\in \dot{B}^{s_{p_1}+\delta_1}_{p_1,p_1}(\mathbb{R}^3)\cap\dot{B}^{s_{p}}_{p,\infty}(\mathbb{R}^3)$ and $U^{2,N^{\kappa}}\in L_{2}(\mathbb{R}^3) \cap \dot{B}^{s_{p}}_{p,\infty}(\mathbb{R}^3).$  Furthermore, \begin{equation}\label{baru_1estinfinityrecallrecall}
\|{u}^{1,N}\|_{\dot{B}^{s_{p_0}+\delta}_{p_0,\infty}}^{p_0}\leq N^{p_0-p} \|g\|_{\dot{B}^{s_{p}}_{p,\infty}}^{p},
\end{equation}
\begin{equation}\label{baru_1est.1infinityrecallrecall}
\|{u}^{1,N}\|_{\dot{B}^{s_{p}}_{p,\infty}}\leq  C(\|\mathcal{F}^{-1}\varphi\|_{L_1})\|g\|_{\dot{B}^{s_{p}}_{p,\infty}},
\end{equation}
\begin{equation}\label{baru_1estkappa}
\|{U}^{1,N^{k}}\|_{\dot{B}^{s_{p_1}+\delta_1}_{p_1,p_1}}^{p_1}\leq N^{\kappa(p_1-p)-\beta(p)p} C(p,p_0,\|\mathcal{F}^{-1}\varphi\|_{L_1},\|g\|_{\dot{B}^{s_{p}}_{p,\infty}}),
\end{equation}

\begin{equation}\label{baru_2est.21kappa}
\|{U}^{1,N^{\kappa}}\|_{\dot{B}^{s_{p}}_{p,\infty}}\leq  C(\|\mathcal{F}^{-1}\varphi\|_{L_1})\|g\|_{\dot{B}^{s_{p}}_{p,\infty}},
\end{equation}

\begin{equation}\label{baru_2estkappa}
\|{U}^{2,N^{\kappa}}\|_{L_2}^2\leq N^{\kappa(2-p)-\beta(p)p} C(p,p_0,p_1,\|\mathcal{F}^{-1}\varphi\|_{L_1},\|g\|_{\dot{B}^{s_{p}}_{p,\infty}})  
\end{equation}
and 
\begin{equation}\label{baru_2estalpha}
\|{U}^{2,N^{\kappa}}\|_{\dot{B}^{s_{p}}_{p,\infty}}\leq  C(\|\mathcal{F}^{-1}\varphi\|_{L_1})\|g\|_{\dot{B}^{s_{p}}_{p,\infty}}.
\end{equation}
 
Let $p_{2}=2 \max(p_0,p_{1})$ and $\delta_{2}= \frac{\min(\delta, \delta_{1})}{2}$. 
From Remark \ref{besovremark2} and (\ref{baru_1estinfinityrecallrecall})-(\ref{baru_1est.1infinityrecallrecall}), we have that $u^{1,N}\in \dot{B}^{s_{p_2}}_{p_2,\infty}(\mathbb{R}^3) \cap \dot{B}^{s_{p_2}+\delta}_{p_2,\infty}(\mathbb{R}^3)$ with estimates 
\begin{equation}\label{baru_1estinfinityrecallrecallembed}
\|{u}^{1,N}\|_{\dot{B}^{s_{p_2}+\delta}_{p_2,\infty}}\leq c(p_2)N^{\frac{p_0-p}{p_0}} \|g\|_{\dot{B}^{s_{p}}_{p,\infty}}^{\frac{p}{p_0}}
\end{equation}
and
\begin{equation}\label{baru_1est.1infinityrecallrecallembed}
\|{u}^{1,N}\|_{\dot{B}^{s_{p_2}}_{p_2,\infty}}\leq  C(\|\mathcal{F}^{-1}\varphi\|_{L_1},p_2)\|g\|_{\dot{B}^{s_{p}}_{p,\infty}}.
\end{equation}
we may apply Proposition \ref{interpolativeinequalitybahourichemindanchin} with
$s_{1}= s_{p_2},$
$s_{2}= s_{p_2}+\delta$
and $\theta=1-\frac{\delta_{2}}{\delta} \in ]0,1[.$ In particular this gives for any $f\in \mathcal{S}^{'}_{h}$:
\begin{equation}\label{interpolativeinequalitydelta2}
\|f\|_{\dot{B}^{s_{p_2}+\delta_2}_{p_2,1}}\leq c(p_2,\delta,\delta_{2})\|f\|_{\dot{B}^{s_{p_2}}_{p,\infty}}^{1-\frac{\delta_{2}}{\delta}}\|f\|_{\dot{B}^{s_{p_2}+\delta}_{p,\infty}}^{\frac{\delta_2}{\delta}}.
\end{equation}
From Remark \ref{besovremark2}, we see that $\dot{B}^{s_{p_2}+\delta_2}_{p_2,1}(\mathbb{R}^3)\hookrightarrow \dot{B}^{s_{p_2}+\delta_2}_{p_2,p_2}(\mathbb{R}^3) $. This,  and (\ref{baru_1estinfinityrecallrecallembed})-(\ref{interpolativeinequalitydelta2}) imply that
\begin{equation}\label{u1Nestdelta2}
\|{u}^{1,N}\|_{\dot{B}^{s_{p_2}+\delta_2}_{p_2,p_2}}\leq N^{\frac{\delta_{2}(p_0-p)}{\delta p_0}} C(p_2, p,p_0, \delta, \delta_2,\|\mathcal{F}^{-1}\varphi\|_{L_1},\|g\|_{\dot{B}^{s_{p}}_{p,\infty}}).
\end{equation}
Using identical reasoning, it can also be inferred that
\begin{equation}\label{UNkappaestdelta2}
\|{U}^{1,N^{k}}\|_{\dot{B}^{s_{p_2}+\delta_2}_{p_2,p_2}}\leq N^{\frac{\delta_{2}(\kappa(p_1-p)-\beta(p)p)}{\delta_1p_1}} C(p,p_0,p_2,p_1, \delta, \delta_1,\|\mathcal{F}^{-1}\varphi\|_{L_1},\|g\|_{\dot{B}^{s_{p}}_{p,\infty}}).
\end{equation}
The choice $$\kappa=\frac{1}{p_1-p}\Big(p\beta(p)+ \frac{\delta_1p_1(p_0-p)}{\delta p_0}\Big)$$ 
implies \begin{equation}\label{besovnormestsum}
\|u^{1,N}+U^{1,N^{\kappa}}\|_{\dot{B}^{s_{p_2}+\delta_2}_{p_2,p_2}}\leq N^{\frac{\delta_{2}(p_0-p)}{\delta p_0}}C(p,p_0,p_{2},p_1,\delta,\delta_1,\delta_2,\|\mathcal{F}^{-1}\varphi\|_{L_1},\|g\|_{\dot{B}^{s_{p}}_{p,\infty}}).
\end{equation}
It is also the case that
\begin{equation}\label{besovnormestsum1}
\|u^{1,N}+U^{1,N^{\kappa}}\|_{\dot{B}^{s_{p}}_{p,\infty}}\leq C(\|\mathcal{F}^{-1}\varphi\|_{L_1})\|g\|_{\dot{B}^{s_{p}}_{p,\infty}}.
\end{equation}

To establish the decomposition of Theorem we
  define $ \bar{g}^{N}$ to be the Leray projector applied to $u^{1,N}+U^{1,N^{\kappa}}$ and $\widetilde{g}^{N}$ to be the Leray projector applied to ${U}^{2,N^{\kappa}}$. Note that the Leray projector  is a continuous linear operator on the homogeneous  Besov spaces under consideration.
\end{proof}

\setcounter{equation}{0}
\section{ Weak* Stability of Global Weak $\B(\mathbb R^3)$-Solutions }
\subsection{Apriori Estimates }

Let $L_{s,l}(Q_T)$, $W^{1,0}_{s,l}(Q_T)$, 
$W^{2,1}_{s,l}(Q_T)$ be anisotropic 
(or parabolic) Lebesgue and Sobolev spaces 
with  norms
$$\|u\|_{L_{s,l}(Q_T)}=\Big(\int\limits_0^T\|u(\cdot,t)\|_{L_s}^ldt\Big)^\frac 1l,\quad \|u\|_{W^{1,0}_{s,l}(Q_T)}=\|u\|_{L_{s,l}(Q_T)}+\|\nabla u\|_{L_{s,l}(Q_T)},$$
$$\|u\|_{W^{2,1}_{s,l}(Q_T)}=\|u\|_{L_{s,l}(Q_T)}+\|\nabla u\|_{L_{s,l}(Q_T)}+\|\nabla^2 u\|_{L_{s,l}(Q_T)}+\|\partial_tu\|_{L_{s,l}(Q_T)}.$$
\begin{lemma}\label{integabilitynonlinearity}
Assume that $u\in L_{\infty}(0,T; J)\cap L_{2}(0,T;\stackrel{\circ}J{^1_2})$ and that there exists an $\alpha>0$ such that
\begin{equation}\label{Holderu}
\textrm{ess}\sup_{0<t<T}\frac{\|u(\cdot,t)\|_{L_{2}}^2}{t^{\alpha}}<\infty.
\end{equation}
Let $u_0\in \B$ be divergence free and let $V(x,t):=S(t)u_0$. 
 Then
\begin{equation}\label{nonlinestsemigroup}
V\cdot\nabla V\in L_{2,\frac 5 4}(Q_T),
\end{equation}
\begin{equation}\label{nonlinest}
V\cdot\nabla u+u\cdot\nabla V\in L_{\frac{3}{2},\frac{6}{5}}(Q_T),
\end{equation}
and
\begin{equation}\label{ONeilest}
V\otimes u:\nabla u
\in L_1(Q_T).
\end{equation}
\end{lemma}
\begin{proof}
By  the H\"older inequality and Proposition \ref{semigroupweakL3}:
$$\|V\cdot\nabla V\|_{L_2}\leq \|V\|_{L_{4}}\|\nabla V\|_{L_{4}}\leq  \frac{c}{t^{\frac 3 4}}\|u_{0}\|_{\B}^2.$$
From here, (\ref{nonlinestsemigroup}) is easily established.
 Again, by the H\"older inequality and Proposition \ref{semigroupweakL3}:
 $$\|u\cdot\nabla V\|_{L_{\frac{3}{2}}}\leq \|\nabla V\|_{L_{6}}\|u\|_{L_{2}}\leq \frac{c\|u\|_{L_{2,\infty}(Q_T)}}{t^{\frac{3}{4}}}\|u_0\|_{\B}.$$
 From this it is immediate that $u\cdot\nabla V\in L_{\frac{3}{2},\frac{6}{5}}(Q_T).$ 
 Using H\"older's inequality once more, one can verify that
 $$\int\limits_0^T\|V\cdot\nabla u\|_{L_{\frac{3}{2}}}^{\frac{6}{5}}dt\leq (\int\limits_0^T \|\nabla u\|_{L_2}^2 dt)^{\frac{3}{5}} (\int\limits_0^T \|V\|_{L_{6}}^3 dt)^{\frac{2}{5}}.$$
 The desired conclusion is reached by noting that Proposition \ref{semigroupweakL3} gives:
$$\|V\|_{L_{6}}^3\leq \frac{c}{t^{\frac{3}{4}}}\|u_0\|_{\B}^3.$$

The last estimate shows why there are difficulties to prove energy estimate for $u$. 

By  Proposition \ref{semigroupweakL3}, (\ref{Holderu}) and  the H\"older inequality :
$$\int\limits_{0}^T\int\limits_{\mathbb R^3}| V\otimes u:\nabla u |dxd\tau\leq \Big(\int\limits_{0}^T\int\limits_{\mathbb R^3}|\nabla u |^2dxd\tau \Big)^{\frac 1 2}\Big(\int\limits_{0}^T\int\limits_{\mathbb R^3}|V\otimes u |^2dxd\tau \Big)^{\frac 1 2}\leq$$$$\leq
\|u_0\|_{\B}\Big(\int\limits_{0}^T\int\limits_{\mathbb R^3}|\nabla u |^2dxd\tau \Big)^{\frac 1 2}\Big(\int\limits_{0}^T\int\limits_{\mathbb R^3} \frac{1}{\tau^{1-\alpha}}\textrm{ess}\sup_{0<s<T}\Big(\frac{\|u(\cdot,s)\|_{L_2}^2}{s^{\alpha}}\Big)dxd\tau \Big)^{\frac 1 2}.$$
\end{proof}

The next statement is a direct consequence of Lemma \ref{integabilitynonlinearity} and coercive estimates of solutions to the Stokes problem.
\begin{lemma}\label{higherderivativeandpressure}
Let $v$ be a global weak $\B$-solution with functions $u$ and $q$ as in Definition \ref{globalBesovinf}. Then
\begin{equation}\label{vdecomp}
(u,q)=\sum_{i=1}^2(u^i,p_{i})
\end{equation}
such that for any finite $T$: \begin{equation}\label{vdecompspaces}
(u^i,\nabla p_{i})\in W^{2,1}_{s_i,l_i}(Q_T)\times L_{s_i,l_i}(Q_T)
\end{equation} and
\begin{equation}\label{integindices}
(s_1,l_1)=( 9/8,3/2), (s_2,l_2)=( 2,5/4), (s_2,l_2)=(3/2,6/5).
\end{equation}
In addition $(u^i, p_{i})$ satisfy the following:
\begin{equation}\label{v_1eqn}
\partial_t u^1-\Delta u^1+\nabla p_{1}= -u\cdot\nabla u,
\end{equation}
\begin{equation}\label{v_2eqn}
\partial_t u^2-\Delta u^2+\nabla p_{2}= -V\cdot\nabla V
\end{equation}
\begin{equation}\label{v_3eqn}
\partial_t u^3-\Delta u^3+\nabla p_{3}= -V\cdot\nabla u-u\cdot\nabla V
\end{equation}
in $Q_\infty$, and 
\begin{equation}\label{freediv}
{\rm div}\,u^i=0
\end{equation}
in $Q_\infty$ for $i=1,2,3$,
\begin{equation}\label{partinitial}
u^i(\cdot,0)=0
\end{equation}
for all $x\in \mathbb R^3$ and $i=1,2,3$.
\end{lemma}

Before the next Lemma let us introduce some notation. Let $u,\,v$ and $u_0$ be as in Definition \ref{globalBesovinf}. Let $u_0= \bar u_0^{N}+\widetilde u^{N}_0$ denote the splitting from Proposition \ref{besovinfinityinterpolationtheorem}. In particular, $4<p_2<\infty$, $0<\delta_{2}<-s_{p_2}$, $\gamma_{1}>0$ and $\gamma_{2}>0$ are such that   for any $N>0$ there exists  weakly divergence free functions 
$\bar{u}_0^{N}\in \dot{B}^{s_{p_2}+\delta_{2}}_{p_2,p_2}(\mathbb{R}^3)\cap \B(\mathbb{R}^3)$ and $\widetilde{u}^{N}_0\in L_2(\mathbb{R}^3)\cap \B(\mathbb{R}^3) $ with
 
\begin{equation}\label{Vdecomp1recall}
u_0= \bar{u}_0^{N}+\widetilde{u}_0^{N},
\end{equation}
\begin{equation}\label{barV_1estrecall}
\|\bar{u}_0^{N}\|_{\dot{B}^{s_{p_2}+\delta_{2}}_{p_2,p_2}}\leq N^{\gamma_{1}}C( \|u_0\|_{\B}),
\end{equation}
\begin{equation}\label{barV_2estrecall}
\|\widetilde{u}_0^{N}\|_{L_2}\leq N^{-\gamma_{2}}C(\|u_0\|_{\B}).
\end{equation} 
Furthermore, 
\begin{equation}\label{barV_1est.1recall}
\|\bar{u}_0^{N}\|_{\B}\leq C( \|u_0\|_{\B}),
\end{equation}
\begin{equation}\label{barV_2est.1recall}
\|\widetilde{u}_0^{N}\|_{\B}\leq C(  \|u_0\|_{\B}).
\end{equation} 
Let us define the following: 
 \begin{equation}\label{V>N}
 \bar V^{N}(\cdot,t):=S(t)\bar u^{N}_0(\cdot,t),
 \end{equation}
 \begin{equation}\label{V<N}
 \tilde V^{N}(\cdot,t):=S(t)\tilde u^{N}_0(\cdot,t)
 \end{equation}
 and
 \begin{equation}\label{w>N}
 w^{N}(x,t):=u(x,t)+\tilde V^{N}(x,t).
 \end{equation}
\begin{lemma}\label{energyinequalitysplitting} 
 In the above notation, we have the following global energy inequality 
 \begin{equation}\label{w>Nenergyineq}
\|w^N(\cdot,t)\|_{L_2}^2+2\int\limits_0^t\int\limits_{\mathbb R^3} |\nabla w^N(x,t')|^2 dxdt'\leq$$$$\leq\|\tilde u_0^N\|_{L_2}^2+ 
2 \int_0^t\int\limits_{\mathbb R^3}(\bar V^N\otimes w^N+\bar V^N\otimes \bar V^N):\nabla w^N dxdt'
\end{equation}
that is valid for positive $N$ and $t$.
\end{lemma}
\begin{proof}
Let us mention that with Lemma \ref{integabilitynonlinearity} in hand, the proof of Lemma \ref{energyinequalitysplitting} follows from very similar reasoning as presented in \cite{barkerser16}-\cite{barkersersverak16}. We provide all the details here for the convenience of the reader.

The first stage is showing that $w^N$ satisfies the local energy inequality. Let us briefly sketch how this can be done. Let $\varphi\in C^\infty_0(Q_\infty)$ be a positive function. Observe that the assumptions in Definition \ref{globalBesovinf} imply that the following function
\begin{equation}\label{crossterm}
t\rightarrow \int\limits_{\Omega}w^N(x,t)\cdot\bar V^N(x,t)\varphi(x,t)dx
\end{equation}
is continuous for all $t\geq 0$.
It is not so difficult to show that this term has the following expression:
\begin{equation}\label{crosstermexpression1}
\int\limits_{\mathbb R^3}w^N(x,t)\cdot \bar V^N(x,t)\varphi(x,t)dx= \int\limits_0^t\int\limits_{\mathbb R^3}(w^N\cdot \bar V^N)(\Delta\varphi+\partial_t\varphi) dxdt'-$$$$-
2\int\limits_0^t\int\limits_{\mathbb R^3}\nabla w^N:\nabla \bar V^N\varphi dxdt'+\int\limits_0^t\int\limits_{\mathbb R^3} \bar V^N\cdot\nabla\varphi q dxdt'+$$$$+
\frac 1 2\int\limits_0^t\int\limits_{\mathbb R^3} (|v|^2-|w^N|^2)v\cdot\nabla\varphi dxdt'-$$$$-\int\limits_0^t\int\limits_{\mathbb R^3}(\bar V^N\otimes w^N+\bar V^N\otimes \bar V^N):\nabla w^N\varphi dxdt'-$$$$-
\int\limits_0^t\int\limits_{\mathbb R^3}(\bar V^N\otimes \bar V^N+\bar V^N\otimes w^N):(w^N\otimes \nabla\varphi)dxdt'.
\end{equation}
It is also readily shown that
\begin{equation}\label{crosstermexpression2}
\int\limits_{\mathbb R^3}|\bar V^N(x,t)|^2\varphi(x,t)dx=\int\limits_0^t\int\limits_{\mathbb R^3}|\bar V^N(x,t')|^2(\Delta \varphi(x,t')+\partial_t\varphi(x,t'))dxdt'-$$$$-2\int\limits_0^t\int\limits_{\mathbb R^3} |\nabla\bar V^N|^2\varphi dxdt'.
\end{equation}
Using (\ref{localenergyinequality}), together with (\ref{crosstermexpression1})-(\ref{crosstermexpression2}), we obtain that for all $t\in ]0,\infty[$ and for all non negative functions $\varphi\in C_{0}^{\infty}(Q_\infty)$:
\begin{equation}\label{localenergyinequalityW>N}
\int\limits_{\mathbb R^3}\varphi(x,t)|w^N(x,t)|^2dx+2\int\limits_{0}^t\int\limits_{\mathbb R^3}\varphi |\nabla w^N|^2 dxdt^{'}\leq$$$$\leq
\int\limits_{0}^{t}\int\limits_{\mathbb R^3}|w^N|^2(\partial_{t}\varphi+\Delta\varphi)+2qw^N\cdot\nabla\varphi +|w^N|^2v\cdot\nabla\varphi dxdt'+$$$$+
2\int\limits_0^t\int\limits_{\mathbb R^3} (\bar V^N\otimes \bar V^N+\bar V^N\otimes w^N):(\nabla w^N\varphi+w^N\otimes\nabla\varphi)dxdt'
\end{equation}
In the next part of the proof, let $\varphi(x,t)=\varphi_{1}(t)\varphi_{R}(x)$. Here, $\varphi_1\in C_{0}^{\infty}(0,\infty)$ and $\varphi_{R}\in C_0^{\infty}(B(2R))$ are positive functions. Moreover, $\varphi_{R}=1$ on $B(R)$, $0\leq\varphi_{R}\leq 1$, 
$$|\nabla \varphi_R|\leq c/R,\,\,\,\,\,\, |\nabla^2\varphi_R|\leq c/R^2.$$
Since $\widetilde u_0^N\in [C_{0,0}^{\infty}(\mathbb R^3)]^{L_{2}(\mathbb R^3)}$, it is obvious that for $\tilde V^N(\cdot,t):=S(t)\widetilde u_0^N(\cdot,t)$ we the energy equality:
\begin{equation}\label{V>Nenergyequality}
\|\tilde V^N(\cdot,t)\|_{L_2}^2+\int\limits_0^t\int\limits_{\mathbb R^3}|\nabla \tilde V^N|^2 dxdt'=\|\widetilde u_0^N\|_{L_2}^2.
\end{equation}
By semigroup estimates, we have  for $2\leq p\leq\infty$, $4\leq q\leq \infty$:
\begin{equation}\label{V>Nsemigroupest}
\|\tilde V^N(\cdot,t)\|_{L_p}\leq \frac{C(p)}{t^{\frac{3}{2}(\frac{1}{2}-\frac{1}{p})}}\|\widetilde u_0^N\|_{L_2},
\end{equation}
\begin{equation}\label{V<Nsemigroupest}
\|\bar V^N(\cdot,t)\|_{L_q}\leq \frac{C(q)}{t^{\frac{3}{2}(\frac{1}{4}-\frac{1}{q})+\frac{1}{8}}}\|\bar u_0^N\|_{\B}.
\end{equation}
Hence, we have $w^N\in C_{w}([0,T]; J)\cap L_{2}(0,T;\stackrel{\circ}{J}{^1_2})$.
Here, $T$ is finite and $C_{w}([0,T]; J)$ denotes continuity with respect to the weak topology.  Using H\"older's inequality and Sobolev embeddings, this implies that 
\begin{equation}\label{wNmultiplicative}
w^N \in L_{p}(Q_{T})\,\,\,\,\,\textrm{for}\,\,\,\,\, 2\leq p\leq {10}/{3}.
\end{equation}
Using these facts, it is obvious that the following limits hold:
$$\lim_{R\rightarrow\infty}\int\limits_{\mathbb R^3}\varphi_R(x)\varphi_1(t)|w^N(x,t)|^2dx+2\int\limits_{0}^t\int\limits_{\mathbb R^3}\varphi_R\varphi_1 |\nabla w^N|^2 dxdt^{'}=$$$$=
\int\limits_{\mathbb R^3}\varphi_1(t)|w^N(x,t)|^2dx+2\int\limits_{0}^t\int\limits_{\mathbb R^3}\varphi_1 |\nabla w^N|^2 dxdt^{'},$$
$$\lim_{R\rightarrow\infty}\int\limits_{0}^t\int\limits_{\mathbb R^3}(|w^N|^2\partial_t\varphi_{1}\varphi_R+2(\bar V^N\otimes w^N+\bar V^N\otimes\bar V^N):\nabla w^N\varphi_{1}\varphi_{R})dxdt'=$$$$=
\int\limits_{0}^t\int\limits_{\mathbb R^3}(|w^N|^2\partial_t\varphi_{1}+2(\bar V^N\otimes w^N+\bar V^N\otimes\bar V^N):\nabla w^N\varphi_{1})dxdt',$$
$$\lim_{R\rightarrow\infty}\int\limits_{0}^t\int\limits_{\mathbb R^3}(|w^N|^2\varphi_{1}\Delta\varphi_R+\varphi_1|w^N|^2 v\cdot\nabla \varphi_{R}+$$$$+2\varphi_1(\bar V^N\otimes w^N+\bar V^N\otimes\bar V^N):(w^N\otimes\nabla\varphi_R))dxdt'=0. $$
Let us focus on the term containing the pressure, namely
$$ \int\limits_0^t\int\limits_{\mathbb R^3}qw^N\cdot\nabla\varphi_R\varphi_1 dxdt^{'}.$$
It is known that the pressure $q$ can be represented as the composition of Riesz transforms $\mathcal{R}$. In particular,
\begin{equation}\label{pressureRiesz}
q=q_1+q_2,
\end{equation}
\begin{equation}\label{pressureRieszq1}
q_{1}=\mathcal{R}_{i}\mathcal{R}_{j}(u_{i}u_j)
\end{equation} 
and
\begin{equation}\label{pressureRieszq2}
q_{2}=\mathcal{R}_{i}\mathcal{R}_{j}(u_{i}V_j+V_i u_j).
\end{equation}
Since $u\in C_{w}([0,T]; J)\cap L_{2}(0,T;\stackrel{\circ}{J}{^1_2})$, it follows from the H\"older inequality and Sobolev embeddings that $u \in L_{p}(Q_{T})$ for $2\leq p\leq {10}/{3}$.
Using this,  Proposition \ref{semigroupweakL3} and  continuity of the Riesz transforms on Lebesgue spaces, we infer that
\begin{equation}\label{pressure spaces}
q_{1} \in L_{\frac 3 2}(Q_{T})\,\,\,\textrm{and}\,\,\,q_{2} \in L_{2}(\mathbb{R}^3 \times ]\epsilon,T[)\,\,\,\textrm{for\,\,any}\,\,\,0<\epsilon<T\,\,\,\textrm{and}\,\,\,T>0.
\end{equation}
From this and (\ref{wNmultiplicative}), we infer that 
$$\lim_{R\rightarrow \infty}\int\limits_0^t\int\limits_{T(R)}qw^N\cdot\nabla\varphi_R\varphi_1 dxdt^{'}=0.$$
Thus, putting everything together, we get for arbitrary positive function $\phi_1\in C_{0}^{\infty}(0,\infty)$:
\begin{equation}\label{w>Nenergyineqcompacttime}
\int\limits_{\mathbb R^3}\varphi_1(t)|w^N(x,t)|^2dx+2\int\limits_{0}^t\int\limits_{\mathbb R^3}\varphi_{1}(t') |\nabla w^N|^2 dxdt^{'}\leq$$$$\leq
\int\limits_{0}^{t}\int\limits_{\mathbb R^3}|w^N|^2\partial_{t}\varphi_{1}+
2(\bar V^N\otimes w^N+\bar V^N\otimes \bar V^N):\nabla w^N\varphi_1 dxdt'
\end{equation}

From Remark  \ref{holderuinteg}, we see that
\begin{equation}\label{w>Nconverginitialdata}
\lim_{t\rightarrow0}\|w^N(\cdot,t)-\tilde u_0^N(\cdot)\|_{L_{2}}=0.
\end{equation}
For $\bar{V}^N$, we have
\begin{equation}\label{VbarupperindexL4}
\|\bar{V}^N(\cdot,t)\|_{L_{4}}\leq \frac{1}{t^{\frac 1 8}}\|\bar{u}^N_0\|_{\B}\leq \frac{1}{t^{\frac 1 8}}C( \|u_0\|_{\B}).
\end{equation}
Using that $\dot{B}^{s_{p_2}+\delta_2}_{p_2,p_2}(\mathbb{R}^3)\hookrightarrow \dot{B}^{-1+\delta_2}_{\infty,\infty}(\mathbb{R}^3)$, together with the heat flow characterisation of homogeneous Besov spaces with negative upper index, we infer that
\begin{equation}\label{Vbarupperindexinfty}
\|\bar{V}^N(\cdot,t)\|_{L_{\infty}}\leq \frac{1}{t^{\frac 1 2-\frac{\delta_2}{2}}}\|\bar{u}^N_0\|_{\dot{B}^{s_{p_2}+\delta_2}_{p_2,p_2}}\leq \frac{N^{\gamma_{1}}}{t^{\frac 1 2-\frac{\delta_2}{2}}}C( \|u_0\|_{\B})
\end{equation}
Thus, we have the following estimates: 
\begin{equation}\label{Barkserest1}
\int\limits_{0}^t\int\limits_{\mathbb R^3}|\bar V^N\otimes w^N:\nabla w^N|dxdt'\leq$$$$\leq
 N^{\gamma_{1}}C( \|u_0\|_{\B})\left(\int\limits_{0}^t\int\limits_{\mathbb R^3} |\nabla w^N|^2dxdt'\right)^{\frac{1}{2}}\left(\int\limits_0^
t\frac{\|w^N(\cdot,\tau)\|^{2}_{L_{2}}}{\tau^{1-\delta_{2}}}d\tau\right)^{\frac{1}{2}}
,
\end{equation}
\begin{equation}\label{Barkserest2}
\int\limits_{0}^t\int\limits_{\mathbb R^3}|\bar V^N\otimes\bar V^N:\nabla w^N| dxdt'\leq Ct^{\frac{1}{4}}C(\|u_0\|_{\B})\left(\int\limits_{0}^t\int\limits_{\mathbb R^3} |\nabla w^N|^2dxdt'\right)^{\frac{1}{2}}.
\end{equation}
Let
$$ \varphi_{\varepsilon}(s):= \left\{ \begin{array}{rl}
\!\!\!\mbox{ $0\,\,\,\,\,\,\,\,\,\,\,\,\, \textrm{if}\,\,\,\, 0\leq s \leq \varepsilon/2$},\\
\mbox{ $ {2(s-(\varepsilon/2))}/{\varepsilon}\,\,\,\,\,\,\,\,\,\,\,\,\, \textrm{if}\,\,\,\, {\varepsilon}/{2}\leq s \leq \varepsilon$},\\
\mbox{$ 1\,\,\,\,\,\,\,\,\,\,\,\,\,\,\,\,\,\,\,\,\,\,\,\,\,\,\, \textrm{if}\,\,\,\, \varepsilon\leq s $}.
\end{array} \right.
$$
Using (\ref{Barkserest1})-(\ref{Barkserest2}) and  by taking suitable approximations of $\varphi_{\varepsilon}$, it can be shown that $\varphi_{1}= \varphi_{\varepsilon}$ is admissible in (\ref{w>Nenergyineqcompacttime}).  From this we obtain that the following inequality is valid for any $\varepsilon>0$ and $t>0$:
\begin{equation}\label{w>Nenergyineqvareepsilon}
\int\limits_{\mathbb R^3}|w^N(x,t)|^2dx+2\int\limits_{0}^t\int\limits_{\mathbb R^3}\varphi_{\varepsilon}(t') |\nabla w^N|^2 dxdt^{'}\leq$$$$\leq
\int\limits_{0}^{t}\int\limits_{\mathbb R^3}|w^N|^2\partial_{t}\varphi_{\varepsilon}+
2(\bar V^N\otimes w^N+\bar V^N\otimes \bar V^N):\nabla w^N\varphi_{\varepsilon} dxdt'
. 
\end{equation}
Using (\ref{w>Nconverginitialdata}), we  can obtain (\ref{w>Nenergyineq}) by letting $\varepsilon$ tend to zero in (\ref{w>Nenergyineqvareepsilon}).

\end{proof}
\begin{itemize}
\item[] \textbf{Proof of Lemma \ref{venergyest}}
\end{itemize}
\begin{proof}
First observe that  $u=w^N- \tilde V^N$. Thus, using (\ref{V>Nenergyequality}) we see that
$$\|u(\cdot,t)\|_{L_{2}}^2+\int\limits_0^t\int\limits_{\mathbb R^3} |\nabla u|^2dxdt'\leq$$$$\leq 2\|\tilde u^N_0\|_{L_{2}}^2+2\|w^N(\cdot,t)\|_{L_{2}}^2+2\int\limits_0^t\int\limits_{\mathbb R^3} |\nabla w^N|^2dxdt'. $$
By (\ref{barV_2estrecall}):
\begin{equation}\label{u0>NL2norm}
\|\tilde u_0^N\|_{L_{2}}^2\leq N^{-2\gamma_{2}}C(\|u_0\|_{\B}).
\end{equation}
From now on, denote 
$$y_N(t):=\|w^N(\cdot,t)\|_{L_{2}}^2.$$
Using (\ref{w>Nenergyineq}), estimates (\ref{Barkserest1})-(\ref{Barkserest2}), (\ref{u0>NL2norm}) and the Young's inequality obtain that
\begin{equation}\label{yNest1}
y_N(t)+\int\limits_0^t\int\limits_{\mathbb R^3} |\nabla w^N|^2dxdt'\leq N^{2\gamma_{1}}C(\|u_0\|_{\B})\left(\int\limits_0^
t\frac{y_{N}(\tau)}{\tau^{1-\delta_{2}}}d\tau\right)+$$$$+(N^{-2\gamma_{2}}+t^{\frac 1 2})C(\|u_0\|_{\B}).
\end{equation}
By Gronwall's Lemma we obtain
\begin{equation}\label{yNest2}
y_N(t)\leq (N^{-2\gamma_{2}}+t^{\frac 1 2})C(\|u_0\|_{\B})\times$$$$\times\exp\Big(N^{2\gamma_{1}}C( \|u_0\|_{\B},\delta_2)t^{\delta_{2}}\Big).
\end{equation}
Hence,
\begin{equation}\label{yNest3}
\|u(\cdot,t)\|_{L_{2}}^2+\int\limits_0^t\int\limits_{\mathbb R^3} |\nabla u|^2dxdt'\leq (N^{-2\gamma_{2}}+t^{\frac 1 2})C( \|u_0\|_{\B},\delta_{2})\times$$$$\times\Big(t^{\delta_{2}}N^{2\gamma_{1}}\exp\Big(N^{2\gamma_{1}}C( \|u_0\|_{\B},\delta_2)t^{\delta_{2}}\Big)+1\Big).
\end{equation}
The conclusion is then easily reached by taking $N=t^{-\kappa}$ with $$0<\kappa<{\delta_{2}}/{2\gamma_{1}}.$$
\end{proof}

\subsection{ Proof of Weak* Stability and Existence of Global Weak $\B(\mathbb R^3)$-Solutions}
Once Lemma \ref{venergyest} and Lemma {\ref{integabilitynonlinearity}} are established, the proof of Theorem \ref{weak stability} is along similar lines to arguments in \cite{barkerser16}-\cite{barkersersverak16}. We present the full details for completeness.
\begin{itemize}
\item[]\textbf{ Proof of Theorem \ref{weak stability}}
\end{itemize} 
\begin{proof}
We have
$$u_{0}^{(k)}\stackrel{*}{\rightharpoonup} u_0$$ in $\B$ and
may assume that 
$$M:=\sup\limits_k\|u_0^{(k)}\|_{\B}<\infty.$$

Firstly, define
$$V^{(k)}(\cdot,t):= S(t)u_0^{(k)}(\cdot,t),\qquad V(\cdot,t):=S(t)u_{0}(\cdot,t).$$
We see that $V^{(k)}$ converges to $V$ on $Q_\infty$ in the sense of distributions.
By Proposition \ref{semigroupweakL3}, we see that
\begin{equation}\label{VkweakL3est}
\|V^{(k)}(\cdot,t)\|_{L^{4}}\leq \frac{CM}{t^{\frac 1 8}},
\end{equation}
\begin{equation}\label{Vksemigroupest}
\|\partial^m_t\nabla^l V^{(k)}(\cdot,t)\|_{L_{r}}\leq \frac{CM}{t^{m+\frac k 2+{\frac 1 2}(1-\frac{3}{r})}}.
\end{equation}
Here $r\in [4,\infty]$.
 For $T<\infty$ and $l\in ]1,\infty[$, we have the compact embedding
$$W^{2,1}_{l}(B(n)\times ]0,T[)\hookrightarrow C([0,T]; L_{l}(B(n))).$$
From this and (\ref{Vksemigroupest}) one immediately infers that for every $n\in\mathbb{N}$ and $l\in ]1,\infty[$:
\begin{equation}\label{V^kstrongconverg}
\partial^m_t\nabla^l V^{(k)}\rightarrow\partial^m_t\nabla^l V\,\,\,{\rm in}\,\,\, C([{1}/{n},n];L_{l}(B(n))).
\end{equation}

From Lemma \ref{venergyest} we have that for any $0<T<\infty$:
\begin{equation}\label{v^kenergybdd}
\sup_{0<t<T}\|u^{(k)}(\cdot,t)\|_{L_{2}}^2+\int\limits_0^T\int\limits_{\mathbb R^3} |\nabla u^{(k)}|^2dxdt'\leq f_{0}(M,T,\beta,\delta_{2}).
\end{equation}
By H\"older's inequality and the Sobolev inequality, this implies that
\begin{equation}\label{multiplicativebound}
\|u^{(k)}\|_{L_{p}(Q_{T})}\leq f_{1}(M,T,\beta,\delta_{2})\,\,\,\,\,\textrm{for\,\,any}\,\,\,\,2\leq p\leq {10}/{3}.
\end{equation}
By means of a Cantor diagonalisation argument, we can find a subsequence such that for any finite $T>0$:
\begin{equation}\label{v_kweak*}
u^{(k)}\stackrel{*}{\rightharpoonup}u\,\,\, {\rm in}\,\,\,L_{2,\infty}(Q_T),
\end{equation}
\begin{equation}\label{gradv_kweak}
\nabla u^{(k)}{\rightharpoonup}\nabla u\,\,\, {\rm in}\,\,\,L_{2}(Q_T). 
\end{equation}
Using (\ref{v_kweak*}), together with (\ref{venergybddscaled}), we also get that for $0<t<1$:
\begin{equation}\label{vzeronearinitialtime}
\|u\|_{L_{2,\infty}(Q_t)}\leq c(M,\delta_{2})t^{\frac{\beta}{2}}.
\end{equation}

From (\ref{v^kenergybdd}) it is easily inferred that
\begin{equation}\label{v^k9/83/2bdd}
\|u^{(k)}\cdot\nabla u^{(k)}\|_{L_{\frac{9}{8},\frac{3}{2}(Q_T)}}\leq
f_{2}(M,T,\beta,\delta_{2}).
\end{equation}
By the same reasoning as in Lemma \ref{integabilitynonlinearity}, we obtain:
\begin{equation}\label{seqnonlin11/7}
\|V^{(k)}\cdot \nabla V^{(k)}\|_{L_{2,\frac 5 4}(Q_T)}\leq f_{3}(M,T),
\end{equation}
\begin{equation}\label{seqnonlin5/43/2}
 \|V^{(k)}\cdot\nabla u^{(k)}+u^{(k)}\cdot\nabla V^{(k)}\|_{L_{\frac{3}{2},\frac{6}{5}}(Q_T)}\leq f_{4}(M,T,\beta,\delta_{2}).
 \end{equation}
 Split $u^{(k)}=\sum_{i=1}^3 u^{i(k)}$ according to Definition \ref{globalBesovinf}, namely (\ref{vdecomp}).
 By coercive estimates for the Stokes system, along with (\ref{v^k9/83/2bdd}) obtain:
 \begin{equation}\label{v^k_1est}
\|u^{1(k)}\|_{W^{2,1}_{\frac 9 8,\frac 3 2}(Q_t)}+\| \nabla p^{(k)}_1\|_{L_{\frac 9 8,\frac 3 2}(Q_T)}\leq Cf_{2}(M,T,\beta,\delta_{2}),
 \end{equation}
 \begin{equation}\label{v^k_2est}
 \|u^{2(k)}\|_{W^{2,1}_{2,\frac 5 4}(Q_T)}+\|\nabla p^{(k)}_2\|_{L_{\frac{11}{7}}(Q_t)}\leq Cf_3(M,T),
 \end{equation}
 \begin{equation}\label{v^k_3est}
 \|u^{3(k)}\|_{W^{2,1}_{\frac 5 4,\frac 3 2}(Q_t)}+\|\nabla p^{(k)}_3\|_{L_{\frac{3}{2},\frac{6}{5}}(Q_T)}
 \leq Cf_{4}(M,T,\beta,\delta_{2}).
 \end{equation}
 By the previously mentioned embeddings, we infer from (\ref{v^k_1est})-(\ref{v^k_3est}) that for any $n\in\mathbb{N}$ we have the following convergence for a certain subsequence:
 \begin{equation}\label{v^kstrongconverg1}
 u^{(k)}\rightarrow u\,\,\,{\rm in}\,\,\, C([0,n]; L_{\frac{9}{8}}(B(n)).
 \end{equation}
 Hence, using (\ref{multiplicativebound}), we infer that for any $s\in ]1,10/3[$
 \begin{equation}\label{v^kstrongconverg2}
 u^{(k)}\rightarrow u\,\,\,{\rm in}\,\,\, L_{s}(B(n)\times ]0,n[).
 \end{equation}
 It is also not so difficult to show that for any $f\in L_{2}$ and for any $n\in\mathbb N$:
 \begin{equation}\label{v^kweakcontconverg}
 \int\limits_{\mathbb R^3} u^{(k)}(x,t)\cdot f(x)dx\rightarrow \int\limits_{\mathbb R^3} u(x,t)\cdot f(x)dx\,\,\,{\rm in}\,\,\,C([0,n]).
 \end{equation}
 Using (\ref{vzeronearinitialtime}) with (\ref{v^kweakcontconverg}), we establish that 
 \begin{equation}\label{vstrongestinitialtime}
 \lim_{t\rightarrow 0}\|u(\cdot,t)\|_{L_{2}}=0.
 \end{equation} 
 Using (\ref{vzeronearinitialtime}), along with the fact that $u \in L_{2,\infty}(Q_T)$ for any $0<T<\infty$, we see that for any $0<T<\infty$
 \begin{equation}\label{uholderestlimit}
 \sup_{0<t<T} \frac{\|u(\cdot,t)\|_{L_2}^2}{t^\beta}<\infty.
 \end{equation}
 All that remains to show is establishing the local energy inequality (\ref{localenergyinequality}) for the limit and establishing the energy inequality (\ref{venergyineq}) for $u$.
 Verifying  the local energy inequality is not so difficult and hence omitted.
 Let us focus on showing (\ref{venergyineq}) for $u$. 
By identical reasoning to Lemma \ref{energyinequalitysplitting}, we have that for an arbitrary positive function $\phi_1(t)\in C_0^{\infty}(0,\infty)$:
\begin{equation}\label{venergyineqcompacttime}
\int\limits_{\mathbb R^3}\varphi_1(t)|u(x,t)|^2dx+2\int\limits_{0}^t\int\limits_{\mathbb R^3}\varphi_{1}(t) |\nabla u|^2 dxdt^{'}\leq$$$$\leq 
\int\limits_{0}^{t}\int\limits_{\mathbb R^3}|u|^2\partial_{t}\varphi_{1}+
2 (V\otimes  u+V\otimes V):\nabla u\varphi_1 dxdt'.
\end{equation}
From (\ref{uholderestlimit}), Lemma \ref{integabilitynonlinearity} and semigroup estimates, we have that 
$$(V\otimes u+V\otimes V):\nabla u\in L_1(Q_T)$$
for any  positive finite $T$.
Using these facts and (\ref{vstrongestinitialtime}), the conclusion  is reached 
in a similar way to the final steps of the proof of Lemma  \ref{energyinequalitysplitting}.
\end{proof}

Let us comment on Corollary \ref{existenseglobal}. Recall that by Proposition \ref{weak*approx}, there exists a solenodial 
 sequence  $u^{(k)}_{0}\in 
 L_{3}(\mathbb R^3)
 $
 such that 
$$u_{0}^{(k)}\stackrel{*}{\rightharpoonup} u_0$$ in $\B$. 
It was shown in \cite{sersve2016} that for any $k$
there exists a global $L_{3}$-weak solution $v^{(k)}$, which satisfies $$\sup_{0<t<\infty}\frac{\|v^{(k)}-S(t)u_{0}^{(k)}\|_{L_2}^2}{t^{\frac 1 2}}<\infty.$$ 
It is also the case that $v^{(k)}$ is a global $\B$ weak solution. Now, Corollary \ref{existenseglobal} follows from Theorem \ref{weak stability}.

\setcounter{equation}{0}
\section{Uniqueness Results}
\subsection{Construction of Strong Solutions}
Before constructing mild solutions we will briefly explain the relevant kernels and their pointwise estimates.

 Let us consider the following Stokes problem:
$$\partial_t v-\Delta v +\nabla q=-\textrm{div}\,F,\qquad {\textrm div }\,v=0$$
in $Q_T$,
$$v(\cdot,0)=0.$$

Furthermore, assume that $F_{ij}\in C^{\infty}_{0}(Q_T).$
 Then a formal solution to the above initial boundary value problem has the form:
$$v(x,t)=\int \limits_0^t\int \limits_{\mathbb R^3}K(x-y,t-s):F(y,s)dyds.$$
The kernel $K$ is derived with the help of the heat kernel $\Gamma$ as follows:
$$\Delta_{y}\Phi(y,t)=\Gamma(y,t),$$
$$K_{mjs}(y,t):=\delta_{mj}\frac{\partial^3\Phi}{\partial y_i\partial y_i\partial y_s}(y,t)-\frac{\partial^3\Phi}{\partial y_m\partial y_j\partial y_s}(y,t).$$
Moreover, the following pointwise estimate is known:
\begin{equation}\label{kenrelKest}
|K(x,t)|\leq\frac{C}{(|x|^2+t)^2}.
\end{equation}
Define
\begin{equation}\label{fluctuationdef}
 G(f\otimes g)(x,t):=\int\limits_0^t\int\limits_{\mathbb{R}^3} K(x-y,t-s): f\otimes g(y,s) dyds.
 \end{equation}
 It what follows, we will use the notation $\pi_{u\otimes u}:= \mathcal{R}_{i}\mathcal{R}_{j}(u_i u_j),$ where $\mathcal{R}$ is a Riesz transform and the summation convention is adopted.
\begin{pro}\label{strongsolution}
 Suppose that $u_0\in \B$ is weakly divergence free. 
There exists a  universal constant $\varepsilon_{3}$ such that if 
\begin{equation}\label{smallnessasmpinfinityB}
\textrm{ess}\sup_{0<t<T}t^{\frac 1 8}\|S(t)u_{0}\|_{L_4}<\varepsilon_{3},
\end{equation}
then 
there exists a   weak $\B$ solution $\widetilde{v}:=V+\widetilde{u}$ to the Navier-Stokes IVP on $Q_{T}$ that satisfies the following properties.
The first property is that $\widetilde{v}$ satisfies  the estimates
\begin{equation}\label{vintegestinfinityB}
\sup_{0<t<T}t^{\frac 1 8}\|\widetilde{v}\|_{L_4}<2\sup_{0<t<T}t^{\frac 1 8}\|V\|_{L_4}:=2M^{(0)} 
\end{equation}
The other property is that for $\lambda\in]0,T[$ and $k=0,1\ldots$, $l=0,1\ldots$:
\begin{equation}\label{vpsmoothinfinityB}
\sup_{(x,t)\in Q_{\lambda,T}}|\partial_{t}^{l}\nabla^{k} \widetilde{v}|+|\partial_{t}^{l}\nabla^{k} \pi_{\widetilde{v}\otimes \widetilde{v}}|\leq c(\lambda,\|u_0\|_{\B},k,l ).
\end{equation}

\end{pro}
\begin{proof}
Let us introduce the space
$$X_{4}(T):=\{f\in \mathcal{S}^{'}(\mathbb{R}^3\times ]0,T[): \textrm{ess}\sup_{0<t<T} t^{\frac{1}{8}}\|f(\cdot,t)\|_{L_{4}(\mathbb{R}^3)}<\infty\}.$$
\begin{equation}\label{katoclass5def}
\| f \|_{X_{4}(T)}:= \textrm{ess}\sup_{0<t<T} t^{\frac 1 8} \|f(\cdot,t)\|_{L_{4}(\mathbb{R}^3)}.
\end{equation}

We use the method of successive iterations. Let us  define the following,
 for $k=1,2,...,$
$$v^{(1)}=V,\qquad v^{(k+1)}=V+u^{(k+1)},
$$
where $u^{(k+1)}:= G(v^{(k)}\otimes v^{(k)})$ solves the following problem
$$\partial_tu^{(k+1)}-\Delta u^{(k+1)}+\nabla q^{(k+1)}=-{\textrm div}\,v^{(k)}\otimes v^{(k)},\,\,\,\textrm{div}\, u^{k+1}=0$$ in 
$Q_T$, 
$$u^{(k+1)}(\cdot,0)=0$$ 
in $\mathbb R^3$. Using (\ref{kenrelKest}), it is easy to check that 
for solutions to the above linear problem the following estimate is
true 
$$\| u^{(k+1)} \|_{X_{4}(T)}\leq c\| v^{(k)} \|_{X_{4}(T)}^2,$$ 

and thus we have 
$$
\| v^{(k+1)} \|_{X_{4}(T)}\leq \|V\|_{X_{4}(T)}+c\| v^{(k)} \|_{X_{4}(T)}^2
$$

for all $k=1,2,...$.   Using arguments in \cite{Kato1984}, one can show that if  $$\|V\|_{X_{4}(T)}<\varepsilon_{3}<\frac{1}{4c},$$ then 
we have 
\begin{equation}\label{1stest}
\|v^{(k)}\|_{X_{4}(T)}<2\|V\|_{X_{4}(T)}  \end{equation}
for all $k=1,2,...$.
 
Furthermore, Kato's arguments also give that there is a $\widetilde{v}=V+\widetilde{u}$ such that
\begin{equation}\label{converg1}
\| v^{(k)}-\widetilde{v}\|_{X_4(T)},\, \|u^{(k)}-\widetilde{u}\|_{X_4(T)}\rightarrow 0,
\end{equation}

We also can exploit our equation, together with the pressure equation, to derive  the following estimate for the energy and pressure:
\begin{equation}\label{energypresestimate}
\|u^{(k)}-u^{(m)}\|^2_{2,\infty,Q_T}+\|\nabla u^{(k)}-u^{(m)}\|^2_{2,Q_T}+\|\pi_{v^{(k)}\otimes v^{(k)}}-\pi_{v^{(m)}\otimes v^{(m)}}\|^2_{2,Q_T}\leq$$$$\leq c\int\limits^T_0\int\limits_{\mathbb R^3}|v^{(k)}\otimes v^{(k)}-v^{(m)}\otimes v^{(m)}|^2dxdt.
\end{equation}

Using (\ref{converg1}), we immediately see the following
 \begin{equation}\label{u^nstrongconverg}
 u^{(k)}\rightarrow \widetilde{u}\,\,{\rm in}\,\,W^{1,0}_{2}(Q_T)\cap C([0,T];L_2(\mathbb{R}^3)) 
 ,
 \end{equation}
 \begin{equation}\label{presconvergence}
 \pi_{v^{(k)}\otimes v^{(k)}}\rightarrow \pi_{\widetilde{v}\otimes \widetilde{v}}\,\,{\rm in}\,\, L_{2}(Q_T),
 \end{equation}
 \begin{equation}\label{widetildeueqn}
 \partial_t\widetilde{u}-\Delta \widetilde{u}+\nabla \pi_{\widetilde{v}\otimes \widetilde{v}}=-{\textrm div}\,\widetilde{v}\otimes \widetilde{v},\,\,\,\textrm{div}\, \widetilde{u}=0
 \end{equation}
 in $Q_T$ and
 \begin{equation}\label{limnearinitialtime}
 \widetilde{u}(\cdot,0)=0
 .
 \end{equation}
 Furthermore, it is not difficult to show that for $0<t<T$
 \begin{equation}\label{uenergyinequality}
 \|\widetilde{u}(\cdot,t)\|_{L_2}^2 +\int\limits_0^t\int\limits_{\mathbb{R}^3} |\nabla \widetilde{u}(x,t')|^2 dxdt' \leq \int\limits_0^t\int\limits_{\mathbb{R}^3} |\widetilde{v}\otimes \widetilde{v}|^2 dxdt' \leq 4t^{\frac{1}{2}}\|V\|_{X_{4}(T)}^4. 
 \end{equation}
 Thus
 \begin{equation}\label{umildHolder}
 \sup_{0<t<T} \frac{\|\widetilde{u}(\cdot,t)\|_{L_{2}}^2}{t^{\frac 1 2}}<\infty.
 \end{equation}
 Clearly, the pair $\widetilde{v}$ and $ \pi_{\widetilde{v}\otimes \widetilde{v}}$ satisfies the Navier-Stokes equations, in a distributional sense. 
 Showing that $(\widetilde{v},\pi_{\widetilde{v}\otimes \widetilde{v}})$ satisfies the local energy inequality and (\ref{vpsmoothinfinityB}) follows from arguments in \cite{barker2017}\footnote{Specifically Theorem 3.1 of \cite{barker2017}.}. 
 The fact that $\widetilde{v}\otimes \widetilde{v} \in L_{2}(Q_T)$, together with properties of $\widetilde{u}$ from (\ref{converg1}) and (\ref{u^nstrongconverg})-(\ref{limnearinitialtime}), imply that for $t\in [0,T]$:
 \begin{equation}\label{uenergyequalityglobaldef}
 \|\widetilde{u}(\cdot,t)\|_{L_2}^2 +2\int\limits_0^t\int\limits_{\mathbb{R}^3} |\nabla \widetilde{u}(x,t')|^2 dxdt' = 2\int\limits_0^t\int\limits_{\mathbb{R}^3} \widetilde{v}\otimes \widetilde{v}: \nabla\widetilde{u}  dxdt'=$$$$=2\int\limits_0^t\int\limits_{\mathbb{R}^3} {V}\otimes \widetilde{u}: \nabla\widetilde{u}+{V}\otimes {V}: \nabla\widetilde{u}  dxdt'. 
 \end{equation}
 Hence, $\widetilde{u}$ satisfies the energy inequality (on $Q_{T}$) present in our definition of global weak $\B$ solution  (in fact, in this case it is an equality). 
 \end{proof}

 Before proceeding further, we state a known Lemma found in \cite{prodi} and \cite{Serrin}, for example.
\begin{lemma}\label{trilinear}
Let $p\in ]3,\infty]$ and 
\begin{equation}\label{serrinpairs}
\frac{3}{p}+\frac{2}{r}=1.
\end{equation}
Suppose that $w \in L_{p,r}(Q_T)$, $v\in L_{2,\infty}(Q_T)$ and $\nabla v\in L_{2}(Q_T)$. 
Then for $t\in ]0,T[$:
\begin{equation}\label{continuitytrilinear}
\int\limits_0^t\int\limits_{\mathbb{R}^3} |\nabla v||v||w| dxdt'\leq C\int\limits_0^t \|w\|_{L_{p}}^r \|v\|^{2}_{L_2}dt'+\frac{1}{2}\int\limits_0^t\int\limits_{\mathbb{R}^3} |\nabla v|^2 dxdt'.
\end{equation}
\end{lemma}
\subsection{Proof of Proposition \ref{besovsmallness}}
 \begin{proof}
 The proof of Proposition \ref{besovsmallness} is based on ideas developed by the author in \cite{barker2017}.

Suppose, $$\sup_{0<t<T}t^{\frac 1 8}\|S(t)u_{0}\|_{L_{4}}:=\|V\|_{X_{T}}<\varepsilon.$$
Furthermore, suppose $\varepsilon<\varepsilon_{3}$, where $\varepsilon_{3}$ is as in the statement of Proposition \ref{strongsolution}. Then by Proposition \ref{strongsolution}
there exists a weak $\B$ solution ($\widetilde{v}:=\widetilde{u}+V$) on $Q_{T}$  that satisfies
\begin{equation}\label{vintegestinfinityBrecall}
\|\widetilde{v}\|_{X_{4}(T)}<2\|V\|_{X_{4}(T)}.
\end{equation}
Here we recall that
\begin{equation}\label{katoclass5defrecall}
\| f \|_{X_{4}(T)}:= \textrm{ess}\sup_{0<t<T} t^{\frac 1 8} \|f(\cdot,t)\|_{L_{4}(\mathbb{R}^3)}.
\end{equation}
Let us now consider any global $\B$ solution $u:=v+V$, defined on $Q_{\infty}$, with the same initial data $u_0 \in \B(\mathbb{R}^3).$ 
We define 
\begin{equation}\label{wdefB}
w=v-\widetilde{v}=u-\widetilde{u} \in W^{1,0}_{2}(Q_{T})\cap C_{w}([0,T]; J(\mathbb{R}^3)).
\end{equation}
Moreover, $w$ satisfies the following equations
\begin{equation}\label{directsystemwB}
\partial_t w+w\cdot\nabla w+\widetilde{v}\cdot\nabla w+w\cdot\nabla \widetilde{v}-\Delta w=-\nabla q,\qquad\mbox{div}\,w=0
\end{equation}
in $Q_{T}$,
with the initial condition satisfied in the  strong $L_{2}$ sense:
\begin{equation}\label{initialconditionwbmo-1}
\lim_{t\rightarrow 0^{+}}\|w(\cdot,0)\|_{L_{2}}=0
\end{equation}
in $\mathbb{R}^3$.


Using minor adjustments to arguments used in Proposition 14.3 of \cite{LR1}, one can deduce that for $t\in [\delta,T]$: 
\begin{equation}\label{crosstermB}
\int\limits_{\mathbb{R}^3} u(y,t)\cdot \widetilde{u}(y,t)dy=\int\limits_{\mathbb{R}^3} u(y,\delta)\cdot \widetilde{u}(y,\delta)dy-2\int\limits_{\delta}^t\int\limits_{\mathbb{R}^3} \nabla u:\nabla \widetilde{u} dyd\tau+$$$$+\int\limits_{\delta}^t\int\limits_{\mathbb{R}^3} V\otimes V: \nabla(u+\widetilde{u})dyd\tau-\int\limits_{\delta}^t\int\limits_{\mathbb{R}^3}  \widetilde{v}\otimes w:\nabla w dyd\tau +$$$$+\int\limits_{\delta}^t\int\limits_{\mathbb{R}^3} V\otimes\widetilde{u}: \nabla\widetilde{u}+  {V}\otimes u: \nabla{u}dyd\tau
\end{equation}
Using Lemma \ref{trilinear} and (\ref{vintegestinfinityBrecall}), we see that 
\begin{equation}\label{trilineatestwB}
\int\limits_{\delta}^t\int\limits_{\mathbb{R}^3} |\widetilde{v}||w||\nabla w| dyd\tau\leq C\int\limits_{\delta}^t \|\widetilde{v}(\cdot,\tau)\|_{L_{4}}^{8}\|w(\cdot,\tau)\|_{L_{2}}^2 d\tau+$$$$+\frac{1}{2}\int\limits_{\delta}^t\int\limits_{\mathbb{R}^3}|\nabla w|^2 dyd\tau \leq $$$$\leq C\|V\|_{X_{4}(T)}^8\int\limits_{\delta}^t \frac{\|w(\cdot,\tau)\|_{L_{2}}^{2}}{\tau} d\tau+ \frac{1}{2}\int\limits_{\delta}^t\int\limits_{\mathbb{R}^3}|\nabla w|^2 dyd\tau.
\end{equation}
Similarly,
\begin{equation}\label{trilineatestwBu}
\int\limits_{\delta}^t\int\limits_{\mathbb{R}^3} |V||{u}||\nabla {u}| dyd\tau\leq $$$$\leq C\|V\|_{X_{4}(T)}^8\int\limits_{\delta}^t \frac{\|{u}(\cdot,\tau)\|_{L_{2}}^{2}}{\tau} d\tau+ \frac{1}{2}\int\limits_{\delta}^t\int\limits_{\mathbb{R}^3}|\nabla{u}|^2 dyd\tau
\end{equation}
and
\begin{equation}\label{trilineatestwBtildeu}
\int\limits_{\delta}^t\int\limits_{\mathbb{R}^3} |V||\widetilde{u}||\nabla \widetilde{u}| dyd\tau\leq $$$$\leq C\|V\|_{X_{4}(T)}^8\int\limits_{\delta}^t \frac{\|\widetilde{u}(\cdot,\tau)\|_{L_{2}}^{2}}{\tau} d\tau+ \frac{1}{2}\int\limits_{\delta}^t\int\limits_{\mathbb{R}^3}|\nabla \widetilde{u}|^2 dyd\tau.
\end{equation}
  Lemma \ref{venergyest}, applied to both $u$ and $\tilde{u}$, implies that there exists $$\beta(\gamma_1,\gamma_2,\delta_2)>0$$ a such that for $0<t<T$ :
\begin{equation}\label{estnearinitialtimeglobalBu}
\|u(\cdot,t)\|_{L_{2}}^2\leq t^{\beta} c(T, \|u_0\|_{\B}, \delta_2),
\end{equation}
\begin{equation}\label{estnearinitialtimeglobalBtildeu}
\|\widetilde{u}(\cdot,t)\|_{L_{2}}^2\leq t^{\beta} c(T, \|u_0\|_{\B}, \delta_2)
\end{equation}
and
\begin{equation}\label{estnearinitialtimeglobalBw}
\|w(\cdot,t)\|_{L_{2}}^2\leq t^{\beta} c(T, \|u_0\|_{\B}, \delta_2).
\end{equation}

Hence,
\begin{equation}\label{wweightedintimelesbesgue}
\sup_{0<t<T} \frac{\|w(\cdot,t)\|_{L_2}^2}{t^{\beta}}<\infty.
\end{equation}
Then (\ref{estnearinitialtimeglobalBu})-(\ref{estnearinitialtimeglobalBw}) allows us to take $\delta\rightarrow 0$ in (\ref{crosstermB}) infer that for any $t\in [0,T]$ we have
\begin{equation}\label{crosstermuptoinitialtimeB}
\int\limits_{\mathbb{R}^3} u(y,t)\cdot \widetilde{u}(y,t)dy=-2\int\limits_{0}^t\int\limits_{\mathbb{R}^3} \nabla u:\nabla \widetilde{u} dyd\tau+$$$$+\int\limits_{0}^t\int\limits_{\mathbb{R}^3} V\otimes V: \nabla(u+\widetilde{u})dyd\tau-\int\limits_{0}^t\int\limits_{\mathbb{R}^3} \widetilde{v}\otimes w:\nabla w dyd\tau +$$$$+\int\limits_{0}^t\int\limits_{\mathbb{R}^3} V\otimes\widetilde{u}: \nabla\widetilde{u}+ V\otimes u: \nabla{u}dyd\tau
\end{equation}
Recall that $u$ and $\widetilde{u}$ satisfy the following energy inequalities for $0<t<T$:
\begin{equation}\label{energyinequalityurecall}
\|u(\cdot,t)\|_{L_{2}}^2+2\int\limits_{0}^t\int\limits_{\mathbb R^3} |\nabla u(x,t')|^2 dxdt'\leq$$$$\leq 2 \int_{0}^t\int_{\mathbb R^3}(V\otimes u+V\otimes V):\nabla udxdt'
\end{equation}
\begin{equation}\label{energyinequalitytildeurecall}
\|\widetilde{u}(\cdot,t)\|_{L_{2}}^2+2\int\limits_{0}^t\int\limits_{\mathbb R^3} |\nabla \widetilde{u}(x,t')|^2 dxdt'\leq$$$$\leq 2 \int_{0}^t\int_{\mathbb R^3}(V\otimes \widetilde{u}+V\otimes V):\nabla \widetilde{u}dxdt'
\end{equation}
From this and (\ref{crosstermuptoinitialtimeB}), we deduce that for any $ t\in [0,T]$
\begin{equation}\label{wenergyequalityuptoinitialtimebmo-1lesbesgue}
\|w(\cdot,t)\|_{L_2}^2+2\int\limits_{0}^{t}\int\limits_{\mathbb{R}^3}|\nabla w|^2 dyd\tau \leq 2\int\limits_{0}^t\int\limits_{\mathbb{R}^3}\widetilde{v}\otimes w:\nabla w dyd\tau.
\end{equation}
Using (\ref{trilineatestwB}) and (\ref{wweightedintimelesbesgue}) we see that for $t\in [0,T]$: 
\begin{equation}\label{wmainest1B}
\|w(\cdot,t)\|_{L_{2}}^2\leq C\|V\|_{X_4(T)}^8 \int\limits_{0}^t \frac{\|w(\cdot,\tau)\|_{L_{2}}^{2}}{\tau} d\tau\leq $$$$\leq{\frac{C}{\beta}t^{\beta}\|V\|_{X_4(T)}^8}\sup_{0<\tau<T}\Big(\frac{\|w(\cdot,\tau)\|_{L_{2}}^2}{\tau^{\beta}}\Big).
\end{equation}
Thus,
\begin{equation}\label{wmainest2B}
\sup_{0<\tau<T}\Big(\frac{\|w(\cdot,\tau)\|_{L_{2}}^2}{\tau^{\beta}}\Big)\leq{\frac{C'}{\beta}\|V\|_{X_4(T)}^8}\sup_{0<\tau<T}\Big(\frac{\|w(\cdot,\tau)\|_{L_{2}}^2}{\tau^{\beta}}\Big).
\end{equation}

If we assume further that that $$\varepsilon\leq \min\Big(\Big(\frac{\beta}{2C'}\Big)^{\frac 1 8}, \varepsilon_3\Big),$$
then
$$\|V\|_{X_{4}(T)}\leq\Big(\frac{\beta}{2C'}\Big)^{\frac 1 8}.$$
 It then immediately follows that $w=0$ on $Q_{T}$. 

\end{proof}

 \setcounter{equation}{0}
 \section{Existence of solutions with $\dot{B}^{s_p}_{p,\infty}(\mathbb{R}^3)$ initial values}
 
  In what follows we will need the following  two Propositions. The proof of Proposition \ref{regularityinfinity}  can be found in \cite{barker2017} and the proof of Proposition \ref{Dallas} can be found in \cite{dallas}, for example.
\begin{pro}\label{regularityinfinity} 
 Consider $p_2$ and $\delta_2$ such that $4<p_2<\infty$, $\delta_2>0$ and $s_{p_2}+\delta_2<0$. Suppose that $u_0\in \dot{B}^{s_{p_2}+\delta_2}_{p_2,p_2}(\mathbb{R}^3)$ is a divergence free tempered distribution. 
There exists a constant $c=c(p_2)$ such that 
if $0<T<\infty$ is such that
\begin{equation}\label{smallnessasmpinfinity}
4cT^{\frac{\delta_2}{2}}\|u_{0}\|_{\dot{B}^{s_{p_2}+\delta_2}_{p_2,p_2}}<1,
\end{equation}
then 
there exists a $w$, which solves the Navier-Stokes equations (\ref{directsystem})-(\ref{directic}) on $Q_T$ in the sense of distributions and satisfies the following properties.
The first property is that $w$ satisfies the estimate
\begin{equation}\label{vintegestinfinity}
\sup_{0<t<T} t^{\frac{s_{p_2}}{2}-\frac{\delta_2}{2}}\|w(\cdot,t)\|_{L_{p_2}}<2\sup_{0<t<T} t^{\frac{s_{p_2}}{2}-\frac{\delta_2}{2}}\|S(t)u_0\|_{L_{p_2}}:=2M^{(0)} 
\end{equation}
The other property is that for $\lambda\in]0,T[$ and $k=0,1\ldots$, $l=0,1\ldots$:
\begin{equation}\label{vpsmoothinfinity}
\sup_{(x,t)\in Q_{\lambda,T}}|\partial_{t}^{l}\nabla^{k} w|+|\partial_{t}^{l}\nabla^{k} \pi_{w\otimes w}|\leq c(p_2,\delta_2,\lambda,\|u_0\|_{\dot{B}_{p_2,p_2}^{s_{p_2}+\delta_2}},k,l ).
\end{equation}
\end{pro}

\begin{pro}\label{Dallas}
Let $u_0= U_0+V_0,$ with $U_0\in \dot{B}^{s_{p_2}+\delta_2}_{p_2,p_2}(\mathbb{R}^3)$ ($0<\delta_2<-s_{p_2}$) and $V_0\in L_{2}(\mathbb{R}^3)$ being  divergence free tempered distributions. Moreover, assume $0<T<\infty$ is such that
\begin{equation}\label{U0smallness}
4cT^{\frac{\delta_2}{2}}\|U_0\|_{\dot{B}^{s_{p_2}+\delta_2}_{p_2,p_2}}<1.
\end{equation}

  Then there exists a solution to Navier-Stokes IVP in $Q_T$ 
 with the following properties.
 In particular,
\begin{equation}\label{weaksolutionsplittingBpdallas}
v=W+u.
\end{equation}
Here, $W$ is a mild solution to the Navier-Stokes equations from Proposition \ref{regularityinfinity}, with initial data $U_{0}\in\dot{B}^{s_{p_2}+\delta_2}_{p_2,p_2}(\mathbb{R}^3)$, such that
\begin{equation}\label{mildBesovestimatedallas}
\sup_{0<t<T} t^{\frac{p_2}{2}-\frac{\delta_{2}}{2}}\|W(\cdot,t)\|_{L_{p_2}}\leq 2\sup_{0<t<T} t^{\frac{p_2}{2}-\frac{\delta_{2}}{2}}\|S(t)u_0\|_{L_{p_2}}.
\end{equation}
Furthermore, $u\in L_{\infty}(0,T; J)\cap L_{2}(0,T;\stackrel{\circ} J{^1_2})$.  
Additionally,  there exists a $q\in L_{\frac 32, {\rm loc}}(Q_T)$ such that $u$ and $q$  satisfy the perturbed Navier Stokes system in the sense of distributions:
\begin{equation}\label{perturbdirectsystemBpdallas}
\partial_t u+v\cdot\nabla v-W\cdot\nabla W-\Delta u=-\nabla q,\qquad\mbox{div}\,u=0
\end{equation}
in $Q_T$.
Furthermore, for any $w\in L_{2}$:
\begin{equation}\label{vweakcontinuityBpdallas}
t\rightarrow \int\limits_{\mathbb R^3} w(x)\cdot u(x,t)dx
\end{equation}
is a continuous function on  $[0,T]$ and
\begin{equation}\label{initialconditiondallas}
\lim_{t\rightarrow 0^+}\|u(\cdot,t)-V_0\|_{L_{2}}=0.
\end{equation}
Moreover, $u$ satisfies the energy inequality:
\begin{equation}\label{venergyineqBpdallas}
\|u(\cdot,t)\|_{L_{2}}^2+2\int\limits_{0}^t\int\limits_{\mathbb R^3} |\nabla u(x,t')|^2 dxdt'\leq$$$$\leq 2 \int_{0}^t\int_{\mathbb R^3}(W\otimes u):\nabla udxdt'+\|V_0\|_{L_2}^2
\end{equation}
for all $t\in [0,T]$.

Finally, $v$ and $q$ satisfy the local energy inequality.
Namely, for almost every $t\in ]0,T[$ the following inequality holds for all non negative functions $\varphi\in C_{0}^{\infty}(Q_T)$:
\begin{equation}\label{localenergyinequalityBpdallas}
\int\limits_{\mathbb R^3}\varphi(x,t)|v(x,t)|^2dx+2\int\limits_{0}^t\int\limits_{\mathbb R^3}\int\varphi |\nabla v|^2 dxdt^{'}\leq$$$$\leq
\int\limits_{0}^{t}\int\limits_{\mathbb R^3}|v|^2(\partial_{t}\varphi+\Delta\varphi)+v\cdot\nabla\varphi(|v|^2+2q) dxdt^{'}.
\end{equation}

\end{pro}
\begin{itemize}
\item[] \textbf{Proof of Theorem \ref{Bpsolutionexistence}}
\end{itemize}
\begin{proof}
Let $u_0 \in \dot{B}^{s_p}_{p,\infty}$ be divergence free. 
Let $u_0= \bar u_0^{N}+\widetilde u^{N}_0$ denote the splitting from Proposition \ref{besovinfinityinterpolationtheorem}. In particular, $p<p_2<\infty$, $0<\delta_{2}<-s_{p_2}$, $\gamma_{1}:=\gamma_{1}(p)>0$ and $\gamma_{2}:=\gamma_{2}(p)>0$ are such that   for any $N>0$ there exists  weakly divergence free functions 
$\bar{u}_0^{N}\in \dot{B}^{s_{p_2}+\delta_{2}}_{p_2,p_2}(\mathbb{R}^3)\cap \Bp(\mathbb{R}^3)$ and $\widetilde{u}^{N}_0\in L_2(\mathbb{R}^3)\cap \Bp(\mathbb{R}^3) $ with
 
\begin{equation}\label{Vdecomp1recallrecall}
u_0= \bar{u}_0^{N}+\widetilde{u}_0^{N},
\end{equation}
\begin{equation}\label{barV_1estrecallrecall}
\|\bar{u}_0^{N}\|_{\dot{B}^{s_{p_2}+\delta_{2}}_{p_2,p_2}}\leq N^{\gamma_{1}}C(p, \|u_0\|_{\Bp}),
\end{equation}
\begin{equation}\label{barV_2estrecallrecall}
\|\widetilde{u}_0^{N}\|_{L_2}\leq N^{-\gamma_{2}}C(p,\|u_0\|_{\Bp}).
\end{equation} 
Furthermore, 
\begin{equation}\label{barV_1est.1recallrecall}
\|\bar{u}_0^{N}\|_{\Bp}\leq C(p, \|u_0\|_{\Bp}),
\end{equation}
\begin{equation}\label{barV_2est.1recallrecall}
\|\widetilde{u}_0^{N}\|_{\Bp}\leq C(p, \|u_0\|_{\Bp}).
\end{equation} 
From (\ref{barV_1estrecallrecall}) we see that for every $0<T<\infty$ we can choose \\$N(T,p, \delta_2, p_2,\gamma_1, \|u_0\|_{\dot{B}^{s_p}_{p,\infty}})>0$ such that 
$$4c(p_2)T^{\frac{\delta_2}{2}}\|\bar{u}^{N}_0\|_{\dot{B}^{s_{p_2}+\delta_{2}}_{p_2,p_2}}<1.$$
The proof is now completed by means of Proposition \ref{Dallas}.
\end{proof}

\end{document}